\documentclass[11pt]{amsart}
\usepackage{amsmath,amsfonts,amssymb} 
\usepackage{multirow}
\usepackage{amsthm} 
\usepackage{enumerate}
\usepackage{ulem}[normalem]
\usepackage{tikz}
\usepackage{xcolor,verbatim,colortbl}
\usepackage{graphicx,caption,subcaption} 
\usepackage{fullpage}
\usepackage[shortlabels]{enumitem}
\allowdisplaybreaks 

\begingroup
    \makeatletter
    \@for\theoremstyle:=definition,remark,plain\do{%
        \expandafter\g@addto@macro\csname th@\theoremstyle\endcsname{%
            \addtolength\thm@preskip\parskip
            }%
        }
\endgroup

\newtheorem{thm}{Theorem}[section]
\newtheorem{cor}[thm]{Corollary}

\newtheorem{lemma}[thm]{Lemma}

\newtheorem{definition}[thm]{Definition}

\newtheorem{claim}{Claim}[thm]

\newtheorem{obs}[thm]{Observation}

\def\wt{\widetilde}
\newcommand{\N}{\mathbb{N}}
 \newcommand{\ol}{\overline}
\newcommand{\powerset}[1]{\operatorname{Pow}(#1)}

\newcommand{\cDP}{\chi_\mathrm{DP}}
\newcommand{\fDP}{f_\mathrm{DP}}
\DeclareMathOperator{\supp}{Supp}
\newcommand{\mad}{\mathrm{mad}}
\newcommand{\va}{\mathrm{va}}
\newcommand{\lva}{\mathrm{lva}}
\newcommand{\dpva}{\mathrm{dpva}}
\newcommand{\LL}{\mathbf{L}}

\title{\vspace{-0.5in} DP vertex-arboricity of sparse graphs}

\author{Peter Bradshaw}
\address{Department of Mathematics, University of Illinois Urbana-Champaign--Champaign, Urbana, IL 61801, USA}
\email{pb38@illinois.edu}
\thanks{Peter Bradshaw received support from NSF RTG grant DMS-1937241 and an AMS Simons Travel Grant.}

\author{Alexandr Kostochka}
\address{Department of Mathematics, University of Illinois Urbana-Champaign--Champaign, Urbana, IL 61801, USA}
\email{kostochk@illinois.edu}
\thanks{Alexandr Kostochka 
is supported in part  by NSF RTG Grant DMS-1937241.}

\author{Zimu Xiang}
\address{Department of Mathematics, University of Illinois Urbana-Champaign--Champaign, Urbana, IL 61801, USA}
\email{zimux2@illinois.edu}
\thanks{Zimu Xiang 
is supported in part  by NSF RTG Grant DMS-1937241.}

\begin{document}

\vspace{-0.3in}

\begin{abstract}
The \emph{vertex arboricity} $\va(G)$ of a  multigraph $G$ is the minimum number $k$ for which 
$V(G)$ can be partitioned into $k$ subsets, each of which induces an acyclic subgraph of $G$.
By definition, if $\va(G)= k$, then the chromatic number, $\chi(G)$, satisfies 
$k\leq \chi(G)\leq 2k$. Fundamental results by Borodin from 1976 and Bollob\' as and Manvel from 1979
imply an analog of Gallai's lower bound on the number of edges in a $(2k-1)$-critical graph. 
We consider a slight generalization of vertex arboricity in the setting of DP-coloring. 
Using this framework,
we derive lower bounds on the number of edges in graphs critical for vertex arboricity and for list arboricity that are better than Gallai's bound, along with similar bounds in our DP-setting.

\end{abstract}
\maketitle
\section{Introduction}

 \emph{Graphs} in this paper cannot  have multiple edges or loops, and
 \emph{multigraphs} may have multiple edges, but no loops. 
We write $V(G)$ and $E(G)$ for
the vertex set and edge set of a (multi)graph $G$, respectively,
and
we
let $|G|:=|V(G)|$ and $\|G\|:=|E(G)|$.
For subsets $A,B\subseteq V(G)$, let $E_G(A,B)$ be the set of edges $uv$ such that $u\in A$ and $v\in B$.
We write $\|A,B\|:=|E_G(A,B)|$.
For convenience, if $A$ or $B$ contains exactly one vertex $v$, then we may write $E_G(v,B)$ or $E_G(A,v)$ and similarly $\|v,B\|$ or $\|A,v\|$.

 \subsection{Proper coloring and vertex arboricity}

 A \emph{(proper) $k$-coloring} of a multigraph $G$ is a mapping $\varphi\,: \,V(G)\to [1,k]$ such that $\varphi(u)\neq \varphi(v)$ for each $uv\in E(G)$.
The \emph{chromatic number} of $G$,  denoted by $\chi(G)$, is the minimum positive integer $k$ for which $G$ has a proper $k$-coloring.
A graph $G$ is \emph{$k$-colorable} if $\chi(G)\leq k$. 
 For a positive integer $k$, a graph $G$ is \emph{$k$-critical} if $\chi(G)=k$, but every
proper subgraph of $G$ is $(k-1)$-colorable.

Naturally, it is easier to color a (multi)graph when it is ``sparse", i.e., is ``close" to an edgeless graph.
Reasonable measures of 
the
sparsity of a (multi)graph $G$ include maximum degree, \emph{maximum average degree}, 
$\mad(G)=\max_{H\subseteq G}\frac{2|E(H)|}{|V(H)|}$, and its refinement, 
$(a,b)$-sparseness.
  We say that
a multigraph $G$  is {\em$(a,b)$-sparse} for $a>0$ and $b \in \mathbb R$
if  for every $A\subseteq V(G)$
with $|A|\geq 2$, the induced subgraph $G[A]$ has at most $a|A|+b$ edges.
By definition, forests are exactly $(1,-1)$-sparse multigraphs.

Starting from Dirac seventy years ago,  the minimum number $f(n,k)$ of edges in an $n$-vertex $k$-critical graph has been studied. One application
of this study is that if we prove that $f(n,k)> an+b$ for all $n\geq k$, then every $(a,b)$-sparse graph
is $(k-1)$-colorable. Significant results
in the study of $f(n,k)$ were obtained by Dirac~\cite{1957DiracAtheorem,1974Dirac},
Gallai~\cite{1963Gallai1,1963Gallai2}, Krivelevich~\cite{1997Krivelevich,1998Krivelevich}, Kostochka and Stiebitz~\cite{1999KoSt,2003KS} and Kostochka and Yancey~\cite{2014KoYa,2018KoYa}.
In particular, Gallai~\cite{1963Gallai2} proved the following structural result,
which was used later by many others.

Recall that a \emph{Gallai tree} is a connected graph in which every block is a complete graph or an odd cycle, and a~\emph{Gallai forest} is a graph in which every connected component is a Gallai tree.

\begin{thm}[Gallai~\cite{1963Gallai2}]\label{Ga2}
Let $k\geq 4$, and let $G$ be a $k$-critical graph and $B$ be the set of $(k-1)$-vertices in $G$.
Then, $G[B]$ is a Gallai forest.
\end{thm}
Gallai~\cite{1963Gallai1} showed that his 
Theorem~\ref{Ga2}  implies the following lower bound on $f(n,k)$:

\begin{equation}\label{in3}
f(n,k)\geq \left({k-1}+\frac{k-3}{k^2-3}\right)\frac{n}{2}\qquad\mbox{if $k \geq 4$ and $n\geq k+2$}.
\end{equation}

The \emph{vertex arboricity} $\va(G)$ of a  multigraph $G$ is the minimum number $k'$ for which  there is a vertex partition $(V_1,\ldots,V_{k'})$ of $V(G)$ such that the subgraph $G[V_i]$ of $G$ induced by $V_i$
is acyclic for each $1\leq i\leq k'$.   Chartrand, Kronk, and Wall~\cite{1968ChKrWa}  introduced and studied this notion in 1968 under the name \emph{point-arboricity}.
Hakimi and Schmeichel~\cite{1989HaSc} proved that already the decision problem
of
whether a given planar graph has vertex arboricity at most $2$ is NP-complete.
Since each forest is 2-colorable, one can view a $k'$-partition of a multigraph $G$ into forests as a proper coloring of $G$ with color classes $C_1,\ldots,C_{2k'}$, with the additional restriction that for each $1\leq i\leq k'$,
$C_{2i-1}\cup C_{2i}$ induces a forest in $G$.
In particular, 
$\frac 12 \chi(G)\leq \va(G)\leq \chi(G)$ for every multigraph $G$.
Both bounds are tight, as
$\va(K_n)=\lceil\frac{n}{2}\rceil=\lceil\frac{\chi(K_n)}{2}\rceil$ for any $n$, and  the complete $4$-partite graph $K_{5,5,5,5}$ satisfies
$\va(K_{5,5,5,5})=4=\chi(K_{5,5,5,5})$.

Similarly to color-critical graphs, we can define a \emph{va-$k'$-critical multigraph} as a
multigraph $G$ with $\va(G)=k'$ such that every proper subgraph $G'$ of $G$ satisfies $\va(G')\leq k'-1$. This definition leads naturally to the function $f'(n,k')$,
defined as
the minimum number of edges in an $n$-vertex va-$k'$-critical (simple) graph. An easy observation is that
$f'(n,2)=n$ and
that
the only $n$-vertex va-$2$-critical multigraph is the cycle $C_n$. 

A generalization of Brooks' Theorem to degenerate subgraphs due to 
Borodin~\cite{1976Borodin} from 1976 and Bollob\' as and Manvel~\cite{1979BoMa} from 1979 implies
a Gallai-type lower bound for $f'(n,k')$:
\begin{equation}\label{in3'}
f'(n,k')\geq \left({2k'-2}+\frac{2k'-4}{(2k'-1)^2-3}\right)\frac{n}{2}\qquad\mbox{if $k' \geq 3$ and $n\geq 2k'+1$}.
\end{equation}

 \subsection{List coloring and variable degeneracy}  
 An interesting generalization of coloring is \emph{list coloring}.
For a set $X$, let $\powerset{X}$ denote the power set of  $X$, and
let
$f(X)$ denote $\bigcup_{v\in X}f(v)$.
 For a multigraph $G$ and a set of colors $Y$, a \emph{list assignment}
is a function $L \colon V(G) \to \powerset{Y}$.
For each $u \in V(G)$, the set $L(u)$ is called
the \emph{list} of $u$.
A proper coloring $\varphi \colon V(G) \to Y$ is called an \emph{$L$-coloring} if  $\varphi(u) \in L(u)$ for each $u \in V(G)$. 
A multigraph $G$ with a list assignment $L$ is  \emph{$L$-colorable} if it admits an $L$-coloring. 
The \emph{list chromatic number} $\chi_\ell(G)$ of $G$  is the least positive integer $k$ such that $G$ is $L$-colorable whenever $L$ is a list assignment for $G$ with $|L(u)| \geq k$ for all $u \in V(G)$. 
If $L(u) = [1,k]$  for all $u \in V(G)$, then $G$ is $L$-colorable if and only if $G$ is $k$-colorable. Thus, list coloring generalizes proper coloring. In particular, $\chi_\ell(G) \geq \chi(G)$ for all graphs $G$.

We call a multigraph $G$  \emph{list $k$-critical} if $\chi_\ell(G) = k$ and
$\chi_\ell(G') \leq k-1$ for each proper subgraph $G'$ of $G$.
Let $f_\ell(n,k)$   be the minimum number of edges in an $n$-vertex list $k$-critical graph.

 A list assignment $L$ for a graph $G$ is  a \emph{degree-list assignment} if $|L(u)| \geq d_G(u)$ for all $u \in V(G)$.  Borodin~\cite{1979Borodin} and Erd\H os, Rubin, and Taylor~\cite{1980ErRuTa} provided a complete characterization of all graphs that are not $L$-colorable with respect to some degree-list assignment $L$,
 giving a generalization of Theorem~\ref{Ga2}.
	
\begin{thm}[Borodin~\cite{1979Borodin}; Erd\H os, Rubin, and Taylor~\cite{1980ErRuTa}]\label{theo:list_Brooks}
Let $G$ be a connected graph and $L$ be a degree-list assignment for $G$. If $G$ is not $L$-colorable, then $G$ is a Gallai tree; furthermore, $|L(u)| = d_G(u)$ for all $u \in V(G)$ and if $u, v \in V(G)$ are two adjacent non-cut vertices, then $L(u) = L(v)$.
	\end{thm}

Exactly as Theorem~\ref{Ga2} implies~\eqref{in3}, Theorem~\ref{theo:list_Brooks} yields the same lower bound on~$f_\ell(n,k)$. Better lower bounds on~$f_\ell(n,k)$
have been proved by
Kostochka and Stiebitz~\cite{2003KS},  Kierstead and Rabern~\cite{2020KiRa}, 
Cranston and Rabern~\cite{2018CrRa},  Rabern~\cite{2016Rabern,2018Rabern} and
Bradshaw, Choi, Kostochka and Xu~\cite{kgeq5,k=4}. But all these bounds are weaker than the best known bounds on~$f(n,k)$. 

 Naturally, the \emph{list vertex arboricity}, $\lva(G)$, of a  multigraph $G$ is the minimum number $k'$ such that for each list assignment $L$ with $|L(u)|\geq k'$ for all $u\in V(G)$ there
 is a function $\varphi \colon V(G) \to Y=\bigcup_{u\in V(G)}L(u)$ such that  $\varphi(u) \in L(u)$ for each $u \in V(G)$ and $G[\varphi^{-1}(y)]$ is a forest for each $y\in Y$. Then, a multigraph $G$  is lva-\emph{$k'$-critical} if $\lva(G) = k'$ and
$\lva(G') \leq k'-1$ for each proper subgraph $G'$ of $G$.
 Let  $f'_\ell(n,k')$ denote
the minimum number of edges in an $n$-vertex lva-$k'$-critical graph. 
Borodin, Kostochka and Toft~\cite{2000BKT} proved the analog of~\eqref{in3'} for list vertex arboricity:
 \begin{equation}\label{in3''}
f'_\ell(n,k')\geq \left({2k'-2}+\frac{2k'-4}{(2k'-1)^2-3}\right)\frac{n}{2}\qquad\mbox{if $k' \geq 3$ and $n\geq 2k'+1$}.
\end{equation}

This was a partial case of a  result on variable degeneracy of (multi)graphs. 
To define it, we introduce some notions. Given an ordering $\sigma=(v_1,\ldots,v_n)$ of vertices of a multigraph $G$, we denote $d^-_\sigma(v_i):=\|\{v_1,\ldots,v_{i-1}\},v_i\|$ and
 $d^+_\sigma(v_i):=\|\{v_{i+1},\ldots,v_{n}\},v_i\|$.

For a positive integer $d$, one of the definitions of a \emph{$d$-degenerate} (respectively, a \emph{strictly 
$d$-degenerate}) multigraph is that this is a multigraph $G$ for which there is an ordering $\sigma=(v_1,\ldots,v_n)$ of vertices of  $G$ such that $d^-_{\sigma}(v_i)\leq d$
(respectively, $d^-_{\sigma}(v_i)< d$) for every $1\leq i\leq n$. In particular, forests are strictly $2$-degenerate graphs. 

 In these terms, vertex arboricity of $G$ is the smallest $k$ such that $V(G)$ can be partitioned into $k$ subsets inducing strictly $2$-degenerate 
multigraphs. Mitchem~\cite{Mitchem} and others considered partitions into $d$-degenerate subgraphs for general $d$. Borodin, Kostochka and Toft~\cite{2000BKT}
considered a more general setting.

Given a function $\ell: V(G)\to \N$, we say that 
a multigraph $G$ is \emph{strictly $\ell$-degenerate} if there is an ordering $\sigma=(v_1,\ldots,v_n)$ of vertices of  $G$ such that $d^-_{\sigma}(v_i)< \ell(v_i)$
 for every $1\leq i\leq n$. 
Let ${\bf L}=(\ell_1,\ldots,\ell_t)$ where each $\ell_i$ is a function from $V(G)$ to the non-negative integers. 
We call this ${\bf L}$ the \emph{list} for $G$ and denote
${\bf \ell}(v)={\bf \ell}_{\bf L}(v)=\sum_{i=1}^t\ell_i(v)$ for all $v\in V(G)$.
We say that $G$ is ${\bf L}$-\emph{partitionable} or $(\ell_1,\ldots,\ell_t)$-\emph{partitionable} if we can partition $V(G)$ into sets $V_1,\ldots,V_t$ so that for every $1\leq i\leq t$ the subgraph $G[V_i]$ is strictly $\ell_i$-degenerate.
List ${\bf L}$ is a \emph{degree-list} for $G$, if 
${\bf \ell}_{\bf L}(v)\geq d(v)$ for all $v\in V(G)$. Borodin et al.~\cite{2000BKT} described all graphs for which there is a degree-list ${\bf L}$ such that $G$ is not 
${\bf L}$-{partitionable}. 
 Schweser,    Stiebitz and Toft~\cite{2024bookStScTo} extended the description to multigraphs.
A partial case of a corollary from these results is~\eqref{in3''}.

{ The study of vertex arboricity and list vertex arboricity is a special case of a more general problem of partitioning the vertices of a graph into subgraphs of bounded degeneracy.
Borodin \cite{1976Borodin} in 1976 and Bollob\'as and Manvel \cite{1979BoMa} in 1979
developed a sufficient condition  for partitioning the vertices of a graph into subgraphs of bounded degeneracy which generalized Brooks' theorem. In 1995, Thomassen \cite{Thomassen1995} proved that the vertex set of every planar graph can be partitioned into a $1$-degenerate subgraph and a $2$-degenerate subgraph, solving a conjecture of Borodin \cite{1976Borodin}.
Thomassen
\cite{Thomassen2001}
later showed that every planar graph can be vertex-partitioned into an independent set and a $3$-degenerate graph.
{
Borodin and Glebov~\cite{BG1} in 2001 proved that the vertex set of every planar graph of girth at least $5$ can be partitioned into a $0$-degenerate subgraph (i.e. an independent set) and a $1$-degenerate subgraph. This result was rediscoverd by
 Kawarabayashi and Thomassen
\cite{KT} in 2009.}

}

\subsection{DP coloring and DP vertex arboricity}
In order to attack an open problem on list colorings of planar graphs,
 Dvo\v r\' ak and Postle~\cite{2018DvPo}  introduced the more general notion of \emph{DP-coloring}
(they called it \emph{correspondence   coloring}).

\begin{definition}\label{def1}
For a multigraph $G$, a \emph{(DP-)cover} of $G$ is a multigraph $H$ whose vertex set is the disjoint union of sets $H(v)$ for $v\in V(G)$  such that 
\begin{itemize}
    \item For each $u \in V(G)$, the set $H(u)$
     is independent  in $H$.
    \item For each $u,v \in V(G)$, if $|E_G(u,v)|=t$, then $E_H(H(u), H(v))$ is the union of $t$ matchings (where each matching is not necessarily perfect and possibly empty). 
\end{itemize}
The vertices of $H$  are called \emph{nodes}
in order to distinguish them from the vertices of $G$. 
\end{definition}

{If $G' \subseteq G$, then we write $H[G']$ for the subgraph of $H$ with vertex set $\bigcup_{u \in V(G')} H(u)$ and edge set formed by the matchings corresponding to the edges of $G'$. Similarly, if $U \subseteq V(G)$, then we write $H[U]$ for the subgraph of $H$ induced by $\bigcup_{u \in U} H(u)$.

An $H$\emph{-coloring} of $G$ is a function
$\phi:V(G)\rightarrow \bigcup_{v\in V(G)}H(v)$
with $\phi(v)\in H(v)$ such that $\{\phi(v) :v\in V(G)\}$ is independent in $H$. 
The \emph{DP chromatic number} $\cDP(G)$ of a multigraph $G$  is the least positive integer $k$ such that $G$ has an $H$-coloring
whenever $H$ is a cover of $G$ with $|H(u)| \geq k$ for all $u \in V(G)$.
Every list coloring problem can be represented as a DP-coloring problem.
In particular, $\cDP(G) \geq \chi_\ell(G)$ for all multigraphs $G$.

A cover $H$ of a multigraph $G$ is  a \emph{degree-cover} if $|H(u)| \geq d_G(u)$ for all $u \in V(G)$. 
We say a multigraph $G$ is \emph{DP degree-colorable} if $G$ has an $H$-coloring whenever $H$ is a degree-cover.

A multigraph $G$ is \emph{DP $k$-critical} if $\cDP(G) = k$ and $\cDP(G') \leq k-1$ for every proper subgraph $G'$ of $G$.
Let $\fDP(n,k)$ denote the minimum number of edges in an $n$-vertex DP $k$-critical  graph. 

Dvo\v r\' ak and Postle~\cite{2018DvPo} described the graphs (and Bernshteyn, Kostochka, and Pron’~\cite{2017BeKoPr} described the multigraphs)
that have no $H$-coloring for some degree cover $H$.
 These descriptions imply that the bound~\eqref{in3} holds also for $\fDP(n,k)$.
Recently, Bradshaw et al.~\cite{kgeq5,k=4} proved somewhat better bounds for 
$\fDP(n,k)$. For $k\geq 5$, their result is as follows.

\begin{thm}[\cite{kgeq5}]\label{thethm} 
    Let $k\geq 5$ and $\lambda  = \left \lceil \frac{k^2 - 7}{2k-7}\right \rceil$.
    If $G$ is an $n$-vertex  DP $k$-critical { (simple)} graph, then $G=K_k$, or
  $|E(G)| \geq \left(k-1+\frac{1}{\lambda }\right)\frac{n}{2}
        +\frac{1}{\lambda }.$  
\end{thm}

 Similarly to the { list vertex arboricity}, the \emph{DP vertex arboricity}, $\dpva(G)$,
  of a  multigraph $G$ is the minimum number $k'$ such that for each cover $H$ with $|H(u)|\geq k'$ for all $u\in V(G)$ there
 is a function $\varphi \colon V(G) \to \bigcup_{u\in V(G)}H(u)$ such that  $\varphi(u) \in H(u)$ for each $u \in V(G)$ and $\{\phi(v): v\in V(G)\}$ induces a forest  in $H$. Then  a multigraph $G$  is dpva\emph{-$k'$-critical} if $\dpva(G) = k'$ and
$\dpva(G') \leq k'-1$ for each proper subgraph $G'$ of $G$.
 Let  $f'_{DP}(n,k')$ denote
the minimum number of edges in an $n$-vertex dpva-$k'$-critical graph. The same lower bound as~\eqref{in3''} for $f'_{DP}(n,k')$ follows from results on variable degeneracy in terms of DP-coloring proved by Lu, Wang and Wang~\cite{LuWW} (for graphs) and by Kostochka, Schweser and Stiebitz~\cite{KSSt} (for multigraphs).

First, this notion was introduced by Nakparasit and Sittitrai~\cite{NakSit}. It extends (and is more complicated than) the notion of ${\bf L}$-{partitionable} multigraphs.

\begin{definition}\label{def2}
A {\em variable cover} for a multigraph $G$, is a pair $(H,{\bf L})$    
such that
\begin{enumerate}[(a)]
\item  $H$ is a DP-cover of $G$ as in Definition~\ref{def1};
     \item 
    For each $v\in V(G)$, there is an associated function $\LL(v): H(v)\to \N_0$ 
    and  $\LL:=\{\LL(v): v\in V(G)\}$.  
   Given a node $(v,c) \in H(v)$, we write $\ell(v,c)$ for the entry of $\LL(v)$ associated with $(v,c)$.  
    \item For each $v \in V(G)$, $\ell(v)=\ell_\LL(v):=\sum_{x\in H(v)}\LL(x)$,
    $\supp(v) =\supp_{\LL}(v):= \{i \in \{1, \dots, k-1\}:\ell(v,i) > 0\}$, and 
     $s(v)=s_{\LL}(v)=|\supp(v)|$.
\end{enumerate}
\end{definition}

We often say that the value $\ell(v,c)$ is the \emph{capacity} of the node $(v,c)$.
Thus, $\supp(v)$ is the set of indices $i$ for which $(v,i)$ has positive capacity.

\begin{definition}\label{def3} For a   variable cover
$(H,\LL)$ of a multigraph $G$, an $(H,\LL)$-{\em coloring} of $G$ is a pair $(\phi,\sigma)$ where $\phi: V(G)\to V(H)$ is such that $\phi(v)\in H(v)$ for each $v\in V(G)$, and 
$\sigma$ is an ordering $\phi(v_1),\ldots,\phi(v_{|G|})$ of $\phi(V(G))$ such that 
 $d_\sigma^-(\phi(v_i))<\ell(v_i)$ for all $v\in V(G)$.
 \end{definition}
 
 A variable cover
$(H,\LL)$ is
a \emph{degree-cover for $G$} if $\ell(v)\geq d(v)$ for all $v\in V(G)$.
A multigraph $G$ is \emph{DPV degree-colorable} if $G$ has an $(H,\LL)$-coloring whenever $(H,\LL)$ is a degree-cover. The descriptions of connected graphs and multigraphs that are not 
DPV degree-colorable  given in~\cite{LuWW} and~\cite{KSSt} 
 imply the bound~\eqref{in3''} for $f'_{DP}(n,k)$.

\subsection{Our results and structure of the paper}
Our main result is improving the bound in~\eqref{in3''} for $f'_{DP}(n,k')$ which would help to also improve  the bound in~\eqref{in3''} for $f'_{\ell}(n,k')$ and 
the bound~\eqref{in3'} for the ordinary vertex arboricity.

\begin{thm}\label{maint} 
    Let $k'\geq 3$ and $\lambda  = \left \lceil \frac{(2k'-1)^2 - 7}{4k'-9}\right \rceil$.
    If $G$ is an $n$-vertex  dpva $k'$-critical { (simple)} graph, then $G=K_{2k'-1}$, or
\begin{equation}\label{t11}
 |E(G)| \geq \left(2k'-2+\frac{1}{\lambda }\right)\frac{n}{2}
        +\frac{1}{\lambda }.    
\end{equation}  
\end{thm}

 For induction purposes, we will prove a refined version of Theorem~\ref{maint} (see the next section) and derive from this version that the bound~\eqref{t11} holds also for  va-$k'$-critical and lva-$k'$-critical graphs:

{
\begin{thm}\label{t-lva}
    Let $k' \geq 3$, and let $G$ be a graph that is 
    { $\lva$-$k'$-critical}. Then, either $G$ is a $K_{2k'-1}$, or 
   $\quad |E(G)| \geq (2k'-2+\frac{1}{\lambda'} ) \frac{|V(G)|}{2} + \frac 1{\lambda'},\quad$
    where $\lambda' = \left \lceil \frac{(2k'-1)^2 - 7}{4k'-9} \right \rceil $.
\end{thm}   

\begin{thm}\label{t-va}
    Let $k' \geq 3$, and let $G$ be a graph that is 
    { $\va$-$k'$-critical}. Then, either $G$ is a $K_{2k'-1}$, or 
   $\quad |E(G)| \geq (2k'-2+\frac{1}{\lambda'} ) \frac{|V(G)|}{2} + \frac 1{\lambda'},\quad$
    where $\lambda' = \left \lceil \frac{(2k'-1)^2 - 7}{4k'-9} \right \rceil $.
\end{thm}  }

In particular, for $n\geq 7$, each of $f'_{DP}(n,3)$, $f'_{\ell}(n,3)$ and 
$f'(n,3)$ is at least $\left(2+\frac{1}{12}\right)n+\frac{1}{6}$, while~\eqref{in3'} implies only $f'(n,3)\geq \left(2+\frac{1}{22}\right)n$.

{
By letting $d = 2k'-2$ in Theorem \ref{maint}, we also obtain the following strengthened version of Brooks' theorem for even $d \geq 4$:
\begin{thm}
    Let $G$ be a graph with no $K_{d+1}$-subgraph.
If the maximum average degree of $G$ is at most $d+ \left \lceil \frac{(d+1)^2 - 7}{2d-5} \right \rceil^{-1}$, then $\dpva(G) \leq \frac d2$.
\end{thm}
}

Although Theorem~\ref{maint} is a claim on simple graphs, for induction reasons we will need a more sophisticated result on multigraphs in a somewhat more general form. The proof uses many ideas from~\cite{kgeq5} and~\cite{KSSt}.

In the next section we introduce more definitions and state the more sophisticated claims implying Theorem~\ref{maint}.  In Section 3, we 
cite the known results that will be used in our arguments 
and set up the proof of the main result.
In Section \ref{sec:propG} we derive helpful properties of a minimum counter-example $G$ to our theorem, including a couple of so-called gap lemmas. In Section \ref{sec:eliminating} we prove that our $G$ does not contain some specific complete subgraphs. In Section \ref{sec:discharging} we use discharging to prove that our counter-example does not exist and in Section \ref{sec:list-arboricity}, we show how a refined version of
Theorem~\ref{maint} implies  Theorem~\ref{t-lva} and~\ref{t-va}.

\section{Restating the main theorem}\label{sec:setup}

We will prove a  statement more general than Theorem~\ref{maint}.
Call a
variable cover $(H,\LL)$ of a multigraph $G$ in the sense of Definition~\ref{def2} $t$-\emph{bounded} if $\ell(\alpha)\leq t$ for every $v\in V(G)$ and every $\alpha\in H(v)$. 
 We also say that this $(H,\LL)$ is a {\em $k$-cover} 
(respectively, a {\em $k^-$-cover}) if
$\ell(v)= k$ (respectively,
$\ell(v)\leq k$) for every $v\in V(G)$.

A multigraph $G$ is $(H,\LL)$-minimal of $G$ has no $(H,\LL)$-coloring, but for every proper subgraph $G'$ of $G$ and $2$-bounded cover $(H,\LL)$ of $G'$ that satisfies $\ell'(v) = \ell(v)$ for each $v \in V(G')$,
$G'$ has an $(H',\LL')$-coloring.

Our measure of sparsity of a multigraph will be a potential function. 
Let $\lambda  = \left \lceil \frac{k^2 - 7}{2k-7}\right 
\rceil$.
For each $v \in V(G)$,  the \emph{potential} of $v$ is 
$$
\rho_{G,\ell }(v)=
\begin{cases}
\ell (v)\lambda +1 & \textrm{if } \ell (v) = k-1, \\

\ell (v)\lambda -1 & \textrm{if } \ell (v)\in\{2,\ldots, k-2\},\\ 

\ell (v)\lambda -2 & \textrm{if } \ell (v) \in \{0,1\},
\end {cases}
$$
and for each pair $xy\in {V(G)\choose 2}$, the {\em potential} is
\[
\rho_{G,\ell}(xy) = 
\begin{cases}
0 & \textrm{if } 

xy\notin E(G),\\
1 - (2\lambda +1)|E_G(x,y)| & 

\textrm{otherwise.}
\end{cases}
\hfill
\]
In other words, if $x$ and $y$ are joined by a single edge in $G$, then $\rho_{G,\ell }(xy) = -2\lambda $, and each additional edge joining $x$ and $y$ adds $-(2\lambda +1)$ to $\rho_{G,\ell }(xy)$. 
For every $A \subseteq V(G)$, 
the \emph{potential} of $A$ is 
\[\rho_{G,\ell }(A) = \sum_{x \in V(A) } \rho_{G,\ell }(x) + \sum_{xy \in  \binom{A}{2} } \rho_{G,\ell }(xy) .\]

In these terms, a new version of our theorem is:

\begin{thm}
\label{thm:v2}
    Let $k \geq 5$.
    Let $G$ be a loopless multigraph and let $(H,\LL)$ be a $2$-bounded variable $(k-1)^-$-cover for $G$.    
     If $G$ is $(H,\LL)$-minimal, then one of the following holds:
    \begin{enumerate}[(a)]
        \item $G=K_{k}$, 
        
        \item $k=5$, and either $G =4K_2$ or $G$ is a double cycle, 
        \item $\rho_{G,\ell} (G) \leq -2$.
    \end{enumerate}
\end{thm}

This result implies Theorem~\ref{maint} as follows. Suppose $G$ is an $n$-vertex  dpva $k'$-critical { (simple)} graph. This means $G$ has a cover $H$ with $|H(u)|=k'-1$ for each $u\in V(G)$ such that there is no $\varphi \colon V(G) \to \bigcup_{u\in V(G)}H(u)$ such that  $\varphi(u) \in H(u)$ for each $u \in V(G)$ and $\{\phi(v): v\in V(G)\}$ induces a forest  in $H$, but for every proper subgraph $G'$ of $G$, such a function $\phi$ exists. Since a graph is a forest exactly when it is strictly $2$-degenerate, this $H$ can be viewed as a variable cover $(H,\LL)$ such that for every node $\alpha\in V(H)$,  $\LL(\alpha) = 2$. Thus  $(H,\LL)$ is a $2$-bounded variable $2(k'-1)$-cover for $G$, and $G$ is $(H,\LL)$-minimal. Since $2(k'-1)+1\geq 5$ and $G$ has no multiple edges, by Theorem~\ref{thm:v2}, either $G=K_{2k'-1}$ or $\rho_{G,\ell} (G) \leq -2$.
The latter means that $n (1+\lambda(2k'-2))-2\lambda |E(G)|\leq -2$, which is equivalent 
to~\eqref{t11}. Thus Theorem~\ref{thm:v2} is more general than Theorem~\ref{maint}.

We refine some definitions and the claim of Theorem~\ref{thm:v2} for a easier proof.
Since we are looking at a $2$-bounded variable $(k-1)^-$-cover, it is convenient for every $v\in V(G)$ to denote each node $\alpha\in H(v)$ as a pair $(v,c)$, where $c\in [k-1]$, and view $\LL(v)$ as a vector $(\ell_1(v),\ldots,\ell_{k-1}(v))$. With this notation, we call the second coordinate, $c$, of $(v,c)$ \emph{a color}. We will prove the following slightly refined form of Theorem~\ref{thm:v2}:

\begin{thm}
\label{thm:v3}
    Let $k \geq 5$.
    Let $G$ be a loopless multigraph and let $(H,\LL)$ be a $2$-bounded variable $(k-1)^-$-cover for $G$.    
     If $G$ is $(H,\LL)$-minimal, then one of the following holds:
    \begin{enumerate}[(a)]
        \item $G=K_{k}$,  and $\ell (v) \in \{k-1, k-2\}$ for each $v \in V(G)$,
        with at most one vertex $w \in V(G)$ satisfying $\ell (w) = k-2$,

        \item $k=5$, and either $G =4K_2$ or $G$ is a double cycle, and $\ell (v) = 4$ for each $v \in V(G)$, or
        \item $\rho_{G,\ell} (G) \leq -2$.
    \end{enumerate}
\end{thm}

{ If the pair $(G,\ell)$ satisfies  (a) or (b), then we refer to $(G,\ell)$ as \textit{exceptional}.}

\section{Useful facts}

Let  $F$ be a multigraph, $\emptyset\neq W\subset V(F)$ and $f$ be an $(H,\LL)$-coloring of 
$F-W$. Define the associated function $\LL_f$ for $F[W]$ as follows: For each $v\in W$ and each $c\in {\rm Supp}_\LL(v)$, if the node $(v,c)$ is connected by exactly $d_f(v,c)$ edges to  $f(V(F)\setminus W)$, then we let
$\ell_{f}(v,c)=\ell(v,c)-d_f(v,c)$ and $\LL_f=(\ell_{f}(v,1),\dots,\ell_{f}(v,k-1))$.

The following two facts are folklore claims on degeneracy of multigraphs.

\begin{lemma}
\label{lem:acyclic}
Let $F$ be a multigraph, and let $(H,\LL)$ be a variable DP-cover of $F$.
If $F$ has a vertex ordering $\tau$ such that 
\begin{equation}\label{order}
 \mbox{$d^-_\tau(v) < \ell(v)$ for each $v \in V(F)$, }   
\end{equation}
 then $F$ has an $(H ,\LL )$-coloring.
\end{lemma}
\begin{proof}
    Suppose that the lemma is false, and let $(F,H ,\LL )$ be a counterexample for which $|F|$ is minimum.
    Consider an ordering $\tau$ of $V(F)$ satisfying~\eqref{order}, and let $v \in V(F)$ be
    the first in $\tau$.
    Since $\ell (v) > d^-_\tau(v) = 0$,
    there exists a node $(v,c) \in H(v)$ with $\ell(v,c) \geq 1$.
    We let $\phi:\{v\} \rightarrow \{c\}$
    be a partial $(H ,\LL )$-coloring of $F$.

    Now, we observe that for each $w \in N(v)$,
    $d^-_{\tau - v}(w) = d^-_\tau(w) - 1 < \ell (w) - 1 \leq \ell_{\phi}(w)$. 
    Similarly, for each $u \in V(F) \setminus N[v]$, $d^-_{\tau-v}(u) = d^-_\tau(u) < \ell (u) = \ell _{\phi}(u)$.
    Therefore, by the minimality of $F$,
    $F-v$ has an $(H ,\LL_{\phi})$-coloring $(\psi,\sigma)$, which can be combined with $\phi$ by placing $\phi$ first to obtain an $(H ,\LL )$-coloring of $F$. This contradicts the initial assumption that $(F,H ,\LL )$ is a counterexample.
\end{proof}

\begin{lemma}
\label{lem:acyc-cor}
Let $F$ be a connected multigraph, and let $(H ,\LL )$ be a variable DP-cover of $F$.
If $\ell(v) \geq d(v)$ for each $v \in V(F)$ and there exists $w \in V(F)$ for which $\ell (w) > d(w)$, then $F$ has an $(H ,\LL )$-coloring.
\end{lemma}
\begin{proof} { Suppose that the lemma is false, and let $(F,H ,\LL )$ be a counterexample for which $|F|$ is minimum.
By Lemma \ref{lem:acyclic}, it is enough to find a vertex ordering $\tau$
 satisfying~\eqref{order}. Let $F_1,\ldots,F_t$ be the components of $F-w$ and for $1\leq i\leq t$ let $w_i$ be a neighbor of $w$ in $F_i$. Then each $F_i$ satisfies the conditions of our lemma with $w_i$ in place of $w$. So by the minimality of $|F|$, each $F_i$ has a vertex ordering $\tau_i$ satisfying~\eqref{order}. Then concatenating $\tau_1,\ldots,\tau_t$ and adding $w$ at the end forms an ordering $\tau$ of $V(F)$ satisfying~\eqref{order}.}
\end{proof}

For a graph $G$ and a positive integer $s$, the {\em multiple} $s G$ of $G$ is the multigraph obtained from $G$ by replacing each edge  $e \in E(G)$
with $s$  edges joining the endpoints of $e$. 
In particular, $1 G=G$.
A \emph{GDP-forest} is a multigraph such that for every block $B$, there exist $n$ and $t$ such that
$B$ is either a $t K_n$ or a $t C_n$.
 (A {\em double cycle} is a multigraph $2 C_n$.)
A {\em GDP-tree} is a connected GDP-forest. 
Note that every Gallai tree is also a GDP-tree.

For positive integers $a,q$ with  $a | q$, a multigraph $G$ and a cover $H$ of $G$, the \emph{$(q,a)$-blowup} of $H$ is a cover of $q G$ obtained by replacing each vertex $u \in V(H)$ with an independent set $I_u$ of size $q/a$, replacing each edge $uw \in E(H)$ with a graph 
$aK_{q/a,q/a}$ joining $I_u$ and $I_w$, and replacing each  $H(v)$ with the set $\bigcup_{u \in H(v)} I_u$. 
 We often write \emph{$q$-blowup} as shorthand for \emph{$(q,1)$-blowup}. If $B$ is a block in a multigraph $G$ with cover $H$ or variable cover $(H,\LL)$, then let $H_B$  denote the multigraph obtained from the subgraph of $H$ induced by $\bigcup_{v\in V(B)}H(v)$
by deleting isolated nodes.

We will need the following theorem proved by Dvo\v r\' ak and Postle~\cite{2018DvPo} for graphs, and by Bernshteyn, Kostochka, and Pron’~\cite{2017BeKoPr} for  multigraphs. The ``moreover" part of the theorem was proved by
Kim and Ozeki~\cite{KimO}.

\begin{thm}[\cite{2018DvPo},~\cite{2017BeKoPr},~\cite{KimO}]
\label{Kim-O}
If a connected multigraph $G$ has no $H$-coloring for a degree-cover $H$, then each block of $G$ is a multiple of a complete graph or a cycle.

Moreover, for every $v\in V(G)$, if $ \mathcal B(v)$ is the set of blocks of $G$
containing $v$, then $H(v)$ can be written as the disjoint union $\bigcup_{B\in  \mathcal B(v)}H_B(v)$
so that $d_B(v)=|H_B(v)|$ for all $B\in  \mathcal B(v)$ and
 the following hold:

(a) If $B=q K_{t}$, then  $H_B$ is a $q$-blowup of the $(t-1)$-cover $H'$ of 
$K_t$ where $H'$ forms 
 $t-1$ disjoint copies 
of $K_t$.

(b) If $B=q C_{2t}$, then  $H_B$ is  a $q$-blowup of a $2$-cover $H'$ of $C_{2t}$ where $H'$ is $C_{4t}$.

(c) If $B=q C_{2t+1}$, then $H_B$ is 
a $q$-blowup of a $2$-cover $H'$ of $C_{2t+1}$ where $H'$ forms
 two disjoint copies of $C_{2t+1}$.  
\end{thm}

The description of connected multigraphs $G$ that have no 
 $(H, \LL)$-coloring for some variable degree-cover $(H, \LL)$ is more complicated for two reasons. First, any block in which $s(v)=1$ and $\ell(v)=d_G(v)$
 for each non-cut vertex $v$ could be in  $G$. 
  Second, the cover $H$ for such $G$ may have multiple edges. But 
 for $2$-bounded  variable covers, the description is somewhat easier to grasp
 { than for the general case}. Below is an adaptation to such covers of the result by Kostochka, Schweser and Stiebitz~\cite[Theorem 8]{KSSt}.

\begin{thm}[\cite{KSSt}]
\label{KSSt}
If a connected multigraph $G$ has no $(H,\LL)$-coloring for a 
$2$-bounded  variable
degree-cover $(H,\LL)$, then 
\begin{enumerate}[(i)]
    \item Each block of $G$ is a multiple of a complete graph or a cycle;
  \item   For each  block $B$ of $G$ and each component $F$ of $H_B$, either $\ell_B(\alpha)=1$ for each $\alpha\in V(F)$ or $\ell_B(\alpha)=2$ for each $\alpha\in V(F)$.
  \item For every $v\in V(G)$, if $ \mathcal B(v)$ is the set of blocks of $G$
containing $v$, then $H(v)$ can be written as the disjoint union $\bigcup_{B\in  \mathcal B(v)}H_B(v)$ so that $d_B(v)=\sum_{\alpha\in H_B(v)}\ell(\alpha)$
for all $B\in  \mathcal B(v)$.
  \item If $B= qK_{t}$, then either  for some nonnegative integers $a,b$
with
$ a+2b= t-1$, 
$H_B$ is the disjoint union of $a$ $q$-blowups of $K_t$
in which every node $(v,c)$ satisfies $\ell(v,c)=1$  and
of $b$ $q$-blowups of $K_t$ in which every node $(v,c)$ satisfies $\ell(v,c)=2$, or
$q$ is even and for some nonnegative integers $a,b_1,b_2$
with
$ a+2b_1+2b_2= t-1$,
$H_B$ is the disjoint union of $a$ $q$-blowups of $K_t$
in which every node $(v,c)$ satisfies $\ell(v,c)=1$, 
of $b_1$ $q$-blowups of $K_t$ in which every node $(v,c)$ satisfies $\ell(v,c)=2$ and of $b_2$ $(q/2)$-blowups of $2K_t$ in which every node $(v,c)$ satisfies $\ell(v,c)=2$.

 \item If $B= qC_{2t}$, then either  
$H_B$ is the  $q$-blowup of $C_{4t}$
in which every node $(v,c)$ satisfies $\ell(v,c)=1$  or
$q$ is even and $H_B$ is the 
 $(q/2)$-blowup of $2C_{2t}$ in which every node $(v,c)$ satisfies $\ell(v,c)=2$.
 
  \item If $B= qC_{2t+1}$, then either  
$H_B$ is the disjoint union of $2$ $q$-blowups of $C_{2t+1}$
in which every node $(v,c)$ satisfies $\ell(v,c)=1$,  or $H_B$ is 
the $q$-blowup of $C_{2t+1}$ in which every node $(v,c)$ satisfies $\ell(v,c)=2$, or
$q$ is even and  $H_B$ is the 
 $(q/2)$-blowup of $2C_{2t+1}$ in which every node $(v,c)$ satisfies $\ell(v,c)=2$.
 
\end{enumerate}
\end{thm}

To state a result on GDP-trees from~\cite{kgeq5}, we need some definitions.

Given $k$, a multigraph $T$ and a function $\ell: V(F) \to \N\cup\{0\}$ with $\ell(v)\geq d_T(v)$ for each $v\in V(T)$, 
let $F=F(T,\ell)$ be the set of vertices $v\in V(T)$ such that
$\ell(v)=k-1>d_T(v)$. 
Let $\overline{F}=V(T)\setminus F$. 
Furthermore, for $v\in V(T)$, let $\mu_{\ell}(v)=\ell(v)-d_T(v)$ and $\mu_{\ell}(T)=\sum_{v\in V(T)}\mu_{\ell}(v)$.

For  $u,v\in V(T)$, 
let $m(u,v) = \max\{0, |E_T(u,v)|-1\}$. For every  $v\in V(T)$, let $m(v)=\frac{1}{2}\sum_{u\in N(v)}m(u,v)$, and let

\begin{equation}\label{m(T)}
m(T)=\sum_{\substack{u,v\in V(T) \\ u\neq v}}m(u,v)=\sum_{v\in V(T)}m(v).
 \end{equation}

Let $\nu=\nu(k)=\frac{k-2}{2k-7}$. We are interested in the following parameter: 
\begin{equation}\label{Ph}
\Phi_k(T,\ell):= \nu \mu_\ell (T)+m(T)+|\overline{F}|-|{F}|.
 \end{equation}

{
\begin{lemma}[\cite{kgeq5}, Lemma 3.2]\label{GDP'} Let $k\geq 5$ and $\nu=\frac{k-2}{2k-7}$.
Suppose that $T$ is a GDP-tree and $\ell: V(T)\to \N\cup\{0\}$ satisfies

(i) $3\leq \ell(v)\leq k-1$ for each $v\in V(T)$,

(ii) $\ell(v)\geq d_T(v)$ for each $v\in V(T)$, and

(iii) $T$ has neither $(k-1)$-regular nor $(k-2)$-regular blocks.

 Then 
 \begin{equation}\label{s(T)'}
\Phi_k(T)>1+\nu.
 \end{equation}
\end{lemma}}

\section{Properties of a minimum counterexample $G$}\label{sec:propG}

For the rest of the paper, we fix a counterexample $(G,H,\LL)$ to Theorem~\ref{thm:v3}, where $G$ is a loopless multigraph 
and $H$ is a $2$-bounded $(k-1)^-$-cover for $G$
  such that (i) $G$ does not have an $(H,\LL)$-coloring, (ii) every proper subgraph $G'$ of $G$ for any $2$-bounded variable cover $(H',\LL')$ with $\ell'(v)=\ell(v)$ for every $v\in V(G')$ admits an
  $(H',\LL')$-coloring and (iii) each $(G'',H'',\LL'')$ with $|G''|<|G|$ satisfies the theorem.
  
This means that
\begin{equation}\label{minG}
\parbox{14cm}{ (a) $G$ is not $(H,\LL)$-colorable; (b) for every smaller multigraph $G'$ and every $(k-1)^-$-cover $H'$ and $2$-bounded list $\LL'$ of $G'$, the triple $(G',H',\LL')$ satisfies the theorem; (c) $\rho_{G,\ell}(V(G))\geq -1$; (d) $K_{k}\not\subseteq G$  and (e) when $k=5$, $G$ contains neither $4K_2$ nor a double cycle.  } 
\end{equation}

For every $uv\in E(G)$, we assume $E_H(H(u), H(v))$ is the union of $|E_G(u, v)|$ maximal matchings.

\subsection{General observations}

We make some observations about our minimum counterexample $(G,H)$ and the potentials of its subgraphs.

\begin{lemma}
\label{lem:partial}
If $W \subseteq V(G)$ and $f$ are defined as above, then $G[W]$ has no $(H,L_f)$-coloring.
\end{lemma}

\begin{proof}
    Suppose that $G[W]$ has a $(H,L_f)$-coloring $f'=(\phi',\sigma')$ and $f=(\phi,\sigma)$.
    We may take $\psi=\phi\cup\phi'$ and $\tau=\sigma\sigma'$, then $(\psi,\tau)$ is an $(H,\LL)$-coloring of $G$, contradicting the minimality of $G$.
\end{proof}

Given an $(H,\LL)$-coloring $(\psi,\tau)$ of a multigraph $F$,
we say that a node $(v,c)$ is  $(\psi,\tau)$-{\em slack} (or {\em slack in $(\psi,\tau)$}, or
just {\em slack} if there is no ambiguity) if
 $\psi(v)=c$, 
$\ell(v,c)=2$ and $d^-_\tau(v,c)=0$. Similarly, we say that
$(v,c)$ is $(\psi,\tau)$-{\em tight} (or {\em tight in $(\psi,\tau)$}, or
just {\em tight})  if
 $\psi(v)=c$ and $(v,c)$ is not slack.

\begin{lemma}
\label{lem:cut-edge}
    Suppose that $S \subseteq V(G)$ satisfies $\|S ,\overline S \| = 1$, and let $v \in S$ be the unique vertex with a neighbor in $\overline S$. 
    Then, there exists a unique node $(v,c) \in H(v)$ such that
    $(v,c)$ is tight in
    every $(H,\LL)$-coloring of $G[S]$. 
\end{lemma}
\begin{proof}
    By the $(H,\LL)$-minimality of $G$, 
    $G[S]$ has an $(H,\LL)$-coloring $(\phi,\sigma)$. 
    We write $\phi(v) = c$.
    Furthermore,
    $G[\overline S]$ has an $(H,\LL)$-coloring $(\psi,\tau)$. 
    We write $w$ for the unique neighbor of $v$ in $\overline S$.
    By Lemma \ref{lem:partial},
    $G[S]$ has no $(H,\LL_{\psi})$-coloring.
    In particular, for every $(H,\LL)$-coloring $(f,\kappa)$ of $G[S]$,
    $\ell(v,f(v)) - 1 \leq \ell_{\psi}(v,f(v)) \leq d_{\kappa}^-(f(v)) < \ell(v,f(v))$.
    As $\ell_{\psi}(v,f(v)) \neq \ell(v,f(v))$ for every choice of $(f,\kappa)$,
    $(v,f(v))$ is always the unique node of $H(v)$ that is adjacent to $(w,\psi(w))$.
    Furthermore, it follows from the inequality chain above that we always have $d^-_{\kappa}(v,f(v)) = \ell(v,f(v))-1$, so that $(v,f(v))$ is always a tight node.
    Since $(\phi,\sigma)$ assigns $\phi(v)=c$, it follows that $f(v) = c$ always, so that
    $(v,c)$ is tight in
    every $(H,\LL)$-coloring of $G$.
\end{proof}

\begin{obs}
\label{obs:zero}
    $G$ has no vertex $v$ for with $\ell (v) = 0$. 
\end{obs}
\begin{proof}
    If $V(G) = \{v\}$, then $\rho_{G,\ell} (G) = \rho_{G,\ell} (v) = -2$, so $G$ is not a counterexample. If $|V(G)| \geq 2$, then $G[\{v\}]$ is a proper subgraph of $G$ that is not  $(H,\LL)$-colorable, contradicting the assumption that $G$ is $(H,\LL)$-minimal.
\end{proof}

\begin{obs}
\label{obs:list-degree}
    For each vertex $v \in V(G)$, $\ell (v) \leq d(v)$.
    
\end{obs}
\begin{proof}
By the $(H,\LL)$-minimality of $G$,
$G-v$ has an $(H,\LL)$-coloring $(\phi,\sigma)$.
By Lemma \ref{lem:partial},
$G[\{v\}]$ has no $(H,\LL_{\phi})$-coloring, so that $\ell_{\phi}(v) = 0$.
As 
\[0 = \ell_{\phi}(v) = \sum_{i=1}^{k-1} \ell_{\phi}(v,i) \geq \sum_{i=1}^{k-1} \left ( 
\ell(v,c) - d_{\phi}(v,c) \right ) \geq \ell(v)-d(v)
,\]
the observation follows.
\end{proof}

We say that a vertex $v \in V(G)$ is \emph{low} if $\ell (v) = d(v)$.

\begin{lemma}
\label{lem:ERT} Let $B$  be the set of low vertices in $G$ and
$B'\subset B$. For any coloring $\phi$ of $G'=G-B'$, any $u\in B'$ and any
$\ell _i$, let $\ell _i(u)-\ell'_i(u)$ be the number of edges connecting  $(u,i)$ to $\phi(G')$. Then 
$(G[B'],H[B'],\LL_{\phi})$ is constructible.
\end{lemma}
\begin{proof} 
By Lemma~\ref{lem:partial}, 
$G[B']$ is not $(H,\LL_{\phi})$-colorable.
Since $(H[B'],\LL_{\phi})$ is a degree cover,  Theorem~\ref{KSSt} yields that
$G[B']$ is $\LL_{\phi}$-constructible.
\end{proof}

\begin{lemma}
\label{lem:j(k-2)}
For each $j \geq 1$, if $S \subsetneq V(G)$ satisfies 
$|E_G(S,\overline S)| \geq j$, then
$\rho_{G,\ell }(S) \geq j(\lambda -2) + 1$. 
{ Furthermore, if $E_G(S,\overline S)$
contains a pair of parallel edges, then $\rho_{G,\ell}(S) \geq j(\lambda-2)+4$.
}
\end{lemma}
\begin{proof}
    Suppose that the lemma is false, and let $j$ be the smallest value for which the lemma does not hold.
    Then, $G$ has a set $S \subsetneq V(G)$ for which $\rho_{G,\ell }(S) \leq j(\lambda -2)$ and
    $|E_G(S,\overline S)| \geq j$, or a set $S \subsetneq V(G)$ for which $\rho_{G,\ell}(S) \leq j(\lambda-2)+3$, $|E_G(S,\overline S)| \geq j$, and $E_G(S,\overline S)$ contains a pair of parallel edges.
    We choose $S$ to be a counterexample to the lemma with largest size.

Since $G$ is $(H,\LL)$-minimal,
$G[S]$ has an $(H,\LL)$-coloring $(\phi,\sigma)$. By Lemma \ref{lem:partial},
$G[\overline S]$ has no $(H,\LL_{\phi})$-coloring.
Therefore, there exists $U \subseteq \overline S$ for which $G[U]$ has a spanning $(H,\LL_{\phi})$-minimal subgraph. 
As $G$ is $(H,\LL)$-minimal, $\LL$ and $\LL_{\phi}$ do not agree on $U$; therefore, $U$ contains a neighbor $u$ of $S$, and $\ell_{\phi}(u) < \ell(u) \leq k-1$.
As $G$ contains no $K_k$ subgraph, $\rho_{G,\LL_{\phi}}(U) \leq -2$.

Now, consider the set $S':=U \cup S$, and write $ q  = |E_G(S,U)|$. As $U$ contains a neighbor of $S$, $ q  \geq 1$. 
 As $\rho_{G,\ell} (v)-\rho_{G',\ell_{\phi}}(v)\leq (\lambda+2)|E_G(v, S)|$ for each $v\in U$, 
$\rho_{G,\ell }(U) \leq -2 +  q  (\lambda + 2)$.
Furthermore, if $E_G(v,S)$ contains a pair of parallel edges, then $\rho_{\ell}(v)-\rho_{\ell_{\phi}}(v) \leq (\lambda+2)|E_G(v,S)|-2$.
We have 
\begin{equation}
\label{eqn:j(k-2)}
\rho_{G,\ell}(S') \leq 
\rho_{G,\ell}(U) + \rho_{G,\ell}(S) - 2\lambda  q  \leq -2 +  q (\lambda + 2) -  q (2\lambda) + \rho_{G,\ell}(S) = -2 -  q (\lambda - 2) + \rho_{G,\ell}(S).
\end{equation}
Furthermore, if $E_G(S,U)$ contains a pair of parallel edges, then
\begin{equation}
\label{eqn:j(k-2)-parallel}
\rho_{G,\ell}(S') \leq 
\rho_{G,\ell}(U) + \rho_{G,\ell }(S) - 2\lambda  q - 1  \leq -2 +  q (\lambda + 2) -2 -  q (2\lambda) - 1 + \rho_{G,\ell}(S) = -5 -  q (\lambda - 2) + \rho_{G,\ell}(S).
\end{equation}

If $j=1$, then $E_G(S,\overline S)$ has no pair of parallel edges;
then, $\rho_{G,\ell}(S) \leq \lambda -2$, so by \eqref{eqn:j(k-2)},
\[\rho_{G,\ell}(S') \leq -2 -  q (\lambda -2) + \lambda -2 = -2 + (1- q )(\lambda -2) \leq -2.\]

Thus, by the maximality of $S$, $S' = V(G)$, contradicting~\eqref{minG}(c). 
So  $j \geq 2$. 
If
$E_G(S,\overline S)$ has no pair of parallel edges,
then
    as $\rho_{G,\ell }(S) \leq j(\lambda -2)$,
    (\ref{eqn:j(k-2)}) implies 
    \begin{equation}
    \label{eqn:S'}
    \rho_{G,\ell}(S')     \leq -2  -  q  (\lambda -2) + j(\lambda -2) 
    =-2 + (j -  q )(\lambda -2).
    \end{equation}
On the other hand, if $E_G(S,\overline S)$ has a pair of parallel edges, then as $\rho_{G,\ell}(S) \leq j(\lambda-2)+3$,
     \eqref{eqn:j(k-2)-parallel}
     implies 
     \[
        \rho_{G,\ell}(S') \leq -5 + q(\lambda-2) + j(\lambda-2)+3=  -2  -  q  (\lambda -2) + j(\lambda -2) 
    =-2 + (j -  q )(\lambda -2),
     \]
     so that \eqref{eqn:S'} holds again.

    If $S' = V(G)$, then by~\eqref{eqn:S'}, $\rho_{G,\ell} (G) \leq -2$, 
 contradicting~\eqref{minG}(c).   
    Otherwise, $|E_G(S', \overline{S'})| \geq 1$, so the $j=1$ case implies that $\rho_{G,\ell}(S') \geq \lambda -1$. In both cases by~\eqref{eqn:S'}, $ q  < j$, so that there is at least one edge in $E_G(S, V(G') \setminus U)$.

    Now, since $|E_G(S,U)| =  q $ and $|E_G(S, \overline S)| \geq j$, we get  $|E_G(S', \overline{S'})| \geq |E_G(S, V(G') \setminus U)| \geq j -  q $.
    As $1 \leq j -  q  \leq j-1$,
     the minimality of $j$
     yields $\rho_{G,\ell}(S') > (j- q )(\lambda -2)$, contradicting (\ref{eqn:S'}). 
\end{proof}

We can make several observations from Lemma~\ref{lem:j(k-2)}.

\begin{obs}
\label{obs:geq-1}
Each vertex subset $U \subseteq V(G)$ satisfies $\rho_{G,\ell }(U) \geq -1$. 
\end{obs}
\begin{proof}
If $1 \leq |U| \leq |V(G) | - 1$, then $\rho_{G,\ell} (U) \geq \lambda -1$ by Lemma~\ref{lem:j(k-2)}. If $|U| = |V(G)|$, then $\rho_{G,\ell} (U) = \rho_{G,\ell} (G) \geq -1$ by~\eqref{minG}(c). 
\end{proof}

\begin{obs}
\label{obs:geqk-1}
    For every nonempty proper subset $S\subsetneq V(G)$, $\rho_{G,\ell}(S) \geq \lambda -1$.
\end{obs}
\begin{proof}
    As $S \subsetneq V(G)$ and $G$ is connected, $|E_G(S, \overline S)| \geq 1$, so the statement follows by setting $j=1$ in Lemma~\ref{lem:j(k-2)}.
\end{proof}

\begin{lemma}
\label{lem:k-1k}
    If $G$ has a nonempty proper subset $S\subsetneq V(G)$ for which $|E_G(S, \overline S)| = 1$, then $\rho_{G,\ell }(S) \in \{\lambda -1,\lambda \}$.
\end{lemma}
\begin{proof}
     Suppose that $E_G(S,\overline S) = \{x_0y_0\}$, where $x_0 \in S$ and $y_0 \in \overline S$.
    We let $(\phi,\sigma)$ be an $(H,\LL)$-coloring of $G[\overline S]$, and by Lemma \ref{lem:partial},
    $G[S]$ has no $(H,\LL_{\phi})$-coloring.
    Therefore, there exists a subset $U \subseteq S$ for which $G[U]$ is $(H,\LL_{\phi})$-minimal, and by the $(H,\LL)$-minimality of $G$, $U$ contains $x_0$.
    Therefore,
    since $G$ has no $K_k$ subgraph,
    $\rho_{G,\ell_{\phi}}(U) \leq -2$ , so that $\rho_{G,\ell}(U) \leq -2 + (\lambda+2) =  \lambda$.
    By a symmetric argument, $\overline S$ has a vertex subset $U'$ containing $y_0$ for which $\rho_{G,\ell}(U') \leq \lambda $.

    Now, $\rho_{G,\ell }(U \cup U') \leq \lambda + \lambda  -2\lambda  =0$. Therefore, Observation~\ref{obs:geqk-1} implies that $V(G) = X \cup Y$, further implying $U = S$ and $\rho_{G,\ell}(S) \leq \lambda $. As Observation~\ref{obs:geqk-1} implies that $\rho_{G,\ell}(S) \geq \lambda -1$, the proof is complete.
\end{proof}

\begin{lemma}
\label{lem:lambda-1}
    If $G$ has disjoint nonempty connected subsets $S_1, S_2 \subseteq V(G)$ for which $S_1 \cup S_2 \neq V(G)$, then $\rho_{G,\ell}(S_i) =  \lambda-1$ for $i \in \{1,2\}$.
\end{lemma}
\begin{proof}
    By the minimality of $G$, for each $i \in \{1,2\}$,
    $G-S_i$ has an $(H,\LL)$-coloring $(\phi_i,\sigma_i)$.
    By Lemma \ref{lem:partial},
    $G-S_i$
    has no $(H,
    \LL_{\phi_{2-i}})$-coloring.
    Therefore, writing $v_i$ for the vertex incident to the boundary of $S_i$, 
    there is a unique node $(v_i,c_i) \in H(v_i)$ such that in every $(H,\LL)$-coloring of $G[S_i]$, $(v_i,c_i)$ is tight.

    Now, let $G'$ be obtained from $G[S_1]$ and $G[S_2]$ by adding the edge $v_1 v_2$.
    Let $H'$ be obtained from $H[S_1 \cup S_2]$ by adding an edge $(v_1,c_1)(v_2,c_2)$. By the discussion above, $G'$ has no $(H',\LL)$-coloring. Therefore, there is a set $U \subseteq V(G')$, containing $v_1$ and $v_2$, for which $\rho_{G',\ell}(U) \leq -2$ (recall that no exceptional graph has a cut-edge).
    Therefore, $\rho_{G,\ell}(U) \leq 2\lambda-2$.

    Now, write $U_i = S_i \cap U$ for $i \in \{1,2\}$.
    By Lemma \ref{lem:j(k-2)}, $\rho(U_i) \geq \lambda-1$.
    Furthermore, as $\rho(U) = \rho_{G,\ell}(U_1) + \rho_{G,\ell}(U_2) \leq 2\lambda-2$, it follows that $\rho_{G,\ell}(U_i) = \lambda-1$.
    As $U_i$ has the single edge of $E_G(S_i,\overline{S_i})$ on its boundary, 
    Lemma \ref{lem:j(k-2)} implies that $U_i = S_i$. Therefore, $\rho_{G,\ell}(S_i) = \lambda-1$.
\end{proof}

\begin{lemma}
\label{lem:no-Kk}
   For vertex subset $U\subseteq V(G)$, if $G[U]$ contains a spanning $K_k^-$, then $G[U]$ is isomorphic to $K_k^-$.
\end{lemma}
\begin{proof}
    Let $K$ be a $K_k$-subgraph in $G[U]$ if any exists.
    If $K$ is not induced or $G[U]$ contains at least $\binom{k}{2}+1$ edges, then $\rho(U) \leq k-2\lambda < -2$, a contradiction.
    If some $v \in U$ satisfies $\ell(v) \leq k-3$, then $\rho(U) \leq k - (2\lambda+2) < -2$, a contradiction. If two vertices $u,v \in U$ satisfy $\ell(u) \leq k-2$ and $\ell(v) \leq k-2$, then $\rho(U) \leq k - 2(\lambda+2) < -2$. 
    Therefore, there is a vertex $u \in U$ such that all $v \in U \setminus\{u\}$ satisfy $\ell(v) = k-1$, and $\ell(u) \geq k-2$.
    Therefore, if $K$ exists, then $K$ is exceptional, a contradiction.

    Now let $x,y\in U$ be the only pair of vertices such that $xy\not\in E(G)$.
    If $G[U]$ contains exactly $\binom{k}{2}$ edges, then $\rho_{H,\ell}\le 0$, so $V(G)=U$.
    Then by Observation~\ref{obs:list-degree}, we may assume that $x$ is the unique vertex such that $\ell(v)=k-1$ for all $v\in U\setminus\{x\}$ and $\ell(x)=d(x)=k-2$.
    In particular $d(y)=\ell(y)=k-1$, so there is some vertex $z\in U\setminus\{x,y\}$ such that $yz$ is the unique edge with multiplicity $2$.
    But then $\rho_\ell(G)\le -\lambda-2<-2$, a contradiction.
    Thus, if $G[U]$ contains a spanning $K_k^-$, then $G[U]$ must be an induced $K_k^-$. 
\end{proof}

\begin{lemma}\label{lem:no1}
    $G$ has no vertex $v$ satisfying $\ell (v) = 1$.
\end{lemma}
\begin{proof}
        If $\ell (v) = 1$, then $\rho_\ell (v) = \lambda -2$, contradicting Observation~\ref{obs:geqk-1}.
\end{proof}

\begin{obs}  
\label{obs:6}
 For every   $v\in V(G)$, $d(v)\leq (k-1)\frac{\lambda}{\lambda-2}$, and if $v$ is incident to  multiple edges, then $d(v)< (k-1)\frac{\lambda}{\lambda-2}$. In particular, if $k=5$ and $v$ is incident to  multiple edges, then $d(v)\leq 5$.
\end{obs}
\begin{proof} Since $\rho_{G,\ell}(v)\leq 1+(k-1)\lambda$, this follows from
  Lemma~\ref{lem:j(k-2)}. The ``in particular" part holds since $\lambda=6$ when $k=5$.    
\end{proof}

\subsection{Properties of $K^-_k$-subgraphs in $G$}

In this subsection, we establish some structural properties of certain $K_k^-$ subgraphs that appear in $G$. We first need some lemmas. The first is an immediate corollary of Theorem~\ref{KSSt}.

\begin{lemma}
\label{lem:clique-crit}
    Let $F$ be a clique, and let $H$ be a cover of $F$. Suppose that $\ell(v) \geq d_F(v)$ for each $v \in V(F)$.
    If $F$ has no $(H,\LL)$-coloring,
    then each component of $H$ is a clique. Furthermore, for each component $J$ of $H$ and color pair $c,c' \in J$, $\ell(c) = \ell(c')$. 
\end{lemma}

Next, we show that under certain conditions, a clique minus an edge has a variable coloring.

{
\begin{lemma}\label{K5-}
    Let $p\geq 5$ and
    $F$ be a $K_p^-$ graph with missing edge $xx'$ and a $2$-bounded  variable cover $(H,\LL)$ such that the following hold:
    \begin{enumerate}[(i)]
        \item $\ell(v)\leq d_F(v)$  for all $v \in V(F)$.
           \item $\sum_{v\in V(F)}(d_F(v)-\ell(v))\leq 1$ or $p\geq 7$ and
           $\sum_{v\in V(F)}(d_F(v)-\ell(v)) \leq 2$.
    \end{enumerate}
    Then, $F$ has an $(H,\LL)$-coloring.
\end{lemma}
\begin{proof} Let $W=V(F)\setminus\{x,x'\}$ and
$Y=\{v\in V(F): d_F(v)>\ell(v)\}$. By (i) and (ii), $|Y|\leq 2$. If $Y\subset W$, then we color the vertices of $Y$ from the lists, and obtain $K_{p-1}^-$ or  $K_{p-2}^-$ with a degree-cover. Since $K_{p-2}^-$ can appear only if $p\geq 7$,
the resulting graph is $2$-connected and not exceptional. So by Theorem~\ref{KSSt}, we can extend the coloring to $F$.

Suppose now one vertex of $Y$ is in $\{x,x'\}$ (say, $x$) and the other in $W$ (say, $y$). Then $p\geq 7$, $\ell(x)=p-3$ and $\ell(y)=p-2$. Since $(H,\LL)$ is $2$-bounded and
$p-3\geq 4,$  $s(x)\geq 2$ and $s(y)\geq 2$. So, we can choose $\phi(y)\in H(y)$ so that
$s_\phi(x)\geq 2$. Using this, we can now choose $\phi(x)$ so that $H_\phi[F-y-x]$ is not the disjoint union  of copies of $K_{p-2}$, and then finish a coloring using Theorem~\ref{KSSt}.

The next possibility is that $Y=\{x\}$ and $p-4\leq \ell(x)\leq p-2$. If $s(x)\geq 2$, then we do as in previous paragraph, so suppose $s(x)=1$ and $H(x)=\{(x,c)\}$. To have $s(x)=1$, we need $\ell(x)\leq 2$, which is possible only if $k=5$ and $\ell(x,c)=\ell(x)=2$.
 Then   $(x,c) $ is adjacent to nodes that induce a triangle in $H$, as when 
        we define the partial coloring $\phi:\{x\} \rightarrow \{c\}$,
        there must be some component $J$ of $H[F-x]$
        such that $\ell_{\phi}(\gamma) = 1$ for every node $\gamma \in J$.
        Therefore, $H$ has a component $J^*$ isomorphic to $K_5^-$
        such that $\ell(\gamma) = 2$
        for every node $\gamma \in J^*$. We choose an induced 
        $K_{1,2}$-subgraph of $F$, which we call $K$,
        and we let $\psi:V(K) \rightarrow J^*$
        be a partial $(H,\LL)$-coloring of $H$.
        Note that $\psi$ exists by Lemma \ref{lem:acyclic} and the fact that $\ell(\gamma) = 2$ for each $\gamma \in J^*$. 
        Then each of the two remaining vertices
        $v \in V(F-K)$
        satisfies $L_{\psi}(v) = [2,0]$ without loss of generality.
        Thus, $F-K$ has an $(H,\LL_{\psi})$-coloring by Lemma \ref{lem:acyclic},
        and so $F$ has an $(H,\LL)$-coloring.

The last possibility is that $Y=\{x,x'\}$, $p\geq 7$ and $\ell(x)=\ell(x')=p-3$.
In particular, $s(x),s(x')\geq 2$. This helps us to show
\begin{equation}\label{515}
   \mbox{\it every component of $H[W]$ is a $K_{p-2}$.} 
\end{equation}
Indeed, if some component $M$ of $H[W]$ is not a $K_{p-2}$, then either it has a node $\gamma$ of degree less than $p-3$, or it is a $(p-3)$-regular graph with more than $p-2$ vertices. In the latter case, choose any node $\gamma$ in this component. In both cases,
 we can choose  $\phi(x)$ and $\phi(x')$ so that $\gamma$ is still in the list $\LL_\phi$, and so is not in a $K_{p-2}$, contradicting Theorem~\ref{KSSt}.
 Thus,~\eqref{515} holds.

 For the same reasons,
\begin{equation}\label{515'}
   \parbox{14.5cm}{\it for every node  $\gamma\in H(x)\cup H(x')$ and every component $M$ of $H[W]$, either $\gamma$ has no neighbors in $M$ or is adjacent to each node in $M$.} 
\end{equation}
By~\eqref{515} and~\eqref{515'}, 
\begin{equation}\label{515''}
   \parbox{14.5cm}{\it for every node  $\gamma\in H(x)\cup H(x')$, there is a component $M(\gamma)$ of $H[W]$ such that $\gamma$ has no neighbors in all other components of $H[W]$.} 
\end{equation}
So, if, say $\gamma\in H(x)$ and the capacities of the nodes in $M(\gamma)$ are smaller than $\ell(\gamma)$, then we first color $F-x$ by degeneracy, and then color $x$ with $\gamma$. Therefore, since $\ell(x)+\ell(x')=2p-6\geq p$, we may assume that there are $\gamma\in H(x)$
and $\gamma'\in H(x')$ with $M(\gamma)=M(\gamma')$. We can rename the colors so that 
the color of each node in $M(\gamma)\cup \{\gamma,\gamma'\}$ is $c$.
Now, if $\ell(\gamma)=1$, then we let $\phi(x)=\phi(x')=c$, and then extend $\phi$ by degeneracy. If $\ell(\gamma)=2$, then we choose   any node $(y,c)\in M(\gamma)$, color
$y,x,x'$ (in this order) with $c$ and again extend the coloring by degeneracy.
\end{proof}}

\begin{lemma}
\label{lem:bd-UB}
For each $S \subsetneq V(G)$, 
if $j = \|S,\overline S\|$, then
$\rho(S) \leq j(\lambda+2)-2$. 
\end{lemma}
\begin{proof}
    Let $S$ be a minimum counterexample to the lemma. 
    If $G[S]$ is disconnected, then $S$ has two proper subsets $S_1$ and $S_2$ for which $\rho(S) = \rho(S_1)+ \rho(S_2)$. 
    By the minimality of $S$,
    $\rho(S_i) \leq (\lambda+2)\|S_i,\overline S\| - 2$ for each $i \in \{1,2\}$, so that 
    $\rho(S) \leq j(\lambda+2)-4$, contradicting the assumption that $S$ is a counterexample. Therefore, we assume that $G[S]$ is connected.

    By the criticality of $G$, $G[\overline S]$ has an $(H,\LL)$-coloring $\phi$. Since $G[S]$ has no $(H,\LL_{\phi})$-coloring, there is $U \subseteq S$, containing at least one vertex with a neighbor in $\overline S$, for which $\rho_{G,\ell_{\phi}}(U) \leq -2$.
    Let $\ell$ be the number of edges of $E_G(S,\overline S)$ incident to $U$. Then,
    $\rho_{G,\ell}(U) \leq -2 + \ell(\lambda+2)$. 
    As $S$ is a counterexample and $\ell \leq j$, $U \subsetneq S$.

    Now, write $U' = U \setminus S$. Let $q = \|U,S \setminus U\|$, and observe that $\|U',\overline U'\| = q + j-\ell$.
    As $G[S]$ is connected, $q \geq 1$. 
    By the minimality of $S$, we know that $\rho(U') \leq (\lambda+2)(q+j-\ell)-2$. Therefore,
    \[
        \rho(S) = \rho(U') + \rho(U) -2q\lambda \leq (\lambda+2)(q+j-\ell)-2 + \ell(\lambda+2) - 2 - 2q\lambda.
    \]
    As $\lambda \geq 2$, this is at most than $j(\lambda+2)-4$, contradicting the assumption that $S$ is a counterexample.
\end{proof}

\begin{lemma}
\label{lem:KK-4}
    Suppose $G$ has a vertex subset $K$ inducing $K_k^-$ with non-adjacent vertices $x,x'$
such that $\ell(v)=k-1$ for each $v \in K$. Then, (i) $\|K,\overline K\| = 4$,  
(ii) $5\leq k\leq 6$, and (iii)
$G-K$ is connected.
    \end{lemma}
\begin{proof}  Let $F=G[K]$.
    For each vertex $v \in K$, write $\tau(v) = d_G(v) - d_F(v)$. 
    Observe $\|K,\overline K\|=\sum \tau(v)$.
    By Observation 
    \ref{obs:list-degree},  
    $d_G(v) \geq \ell(v)= k-1$ for each $v \in K$.
    Therefore, $\tau(x)\geq 1$ and $\tau(x')\geq 1$.

First, since $\rho(K) = k+2 \lambda < 5(\lambda-2)+1$, $\|K,\overline K\| \leq 4$.
     Suppose
    $\sum_{v\in V(K)} \tau(v) { = \|K,\overline K\| }\leq 3$ { or $k\geq 7$.} Let $(\phi,\sigma)$ be an $(H,\LL)$-coloring of $G-K$.
     For each $v \in K$,
    $\ell_{\phi}(v) \geq \ell(v) - (d_G(v) - d_F(v))$, so that 
    \[d_F(v) - \ell_{\phi}(v) \leq d_G(v) - \ell(v) = \tau(v) + d_F(v) - \ell(v),\]
    which is $\tau(v)$ for $v \not \in \{x,x'\}$ and $\tau(v) - 1$ for $v \in \{x,x'\}$.
    Therefore, 
    $$\sum_{v \in K} (d_F(v) - \ell_{\phi}(v)) \leq \sum_{v \in K} \tau(v) - 2 = \|K,\overline K\| - 2\leq \begin{cases}
        1 & \text{ if } 5\leq k \leq  6; \\
        2 &  \text{ if } k\geq 7.
    \end{cases}$$
    Thus,
    the triple $(K,H_K,\LL_\phi)$ 
    satisfies the conditions of Lemma~\ref{K5-}; thus we can extend $(\phi,\sigma)$
    to $K$, contradicting the choice of $(G,H,\LL)$. Otherwise, 
    {  $5\leq k\leq 6$ and  }
    $\sum_{v\in V(K)} \tau(v) { = \|K,\overline K\| = }
     4$.

     This proves (i) and (ii). 

    Now, suppose that $G - K$ is disconnected. Then, there is a subset $S \subseteq V(G) \setminus V(K)$ such that $G[S]$ is a component of $G-K$ for which $\|S,K\| \leq 2$.
    If $\|S, K\| = 1$, then the single edge joining $S$ to $K$ is a cut-edge, so that $\rho_{G,\ell}(S) \leq \lambda$. Then, $\rho_{G,\ell}(K \cup S) \leq k + \lambda < 3(\lambda-2) +1$. Hence, $\|K \cup S, \overline{K \cup S}\| \leq 2$, contradicting the fact that $\|K, \overline K\| = 4$.
    Otherwise, without loss of generality, $V(G)$ can be partitioned into subsets $K,S,T$ where $\|S,K\| = \|T,K\| = 2$.
    Then,
    \[
        -1 \leq \rho_{G,\ell}(V(G)) \leq \rho_{G,\ell}(S) + \rho_{G,\ell}(T) + \rho_{G,\ell}(K) - 8 \lambda.
    \]
    Rearranging, 
    \[
        \rho_{G,\ell}(S) + \rho_{G,\ell}(T) \geq 8\lambda-1 - (k+2\lambda) = 6 \lambda - k - 1.
    \]
    Thus, without loss of generality, 
    \[
        \rho_{G,\ell}(S) \geq  3 \lambda - \frac{k+1}{2}  > 2\lambda + 2,
    \]
where the last inequality follows from the fact that $\lambda=6$ for $k \in \{5,6\}$.
This contradicts Lemma \ref{lem:bd-UB}. 
Thus, (iii) holds.
\end{proof}

\subsection{Other exceptional subgraphs}

\begin{lemma}\label{l42-2}
    If $5\leq k\leq 6$, then $G$ has no $(k-2)K_2$.
\end{lemma}
\begin{proof} Let $5\leq k\leq 6$.
Suppose that $\{u,v\} \subseteq V(G)$ induces a $tK_2$ subgraph with $t \geq k-2$.
If $t \geq k$, then $\rho_{G,\ell}(\{u,v\}) \leq -2$, a contradiction. If $t=4$ and $k=5$,
this contradicts~\eqref{minG}. If $t=5$ and  $k=6$,  then $\rho_{G,\ell}(\{u,v\}) \leq -2$, a contradiction. 
Therefore, $t=k-2$.

{

If $\ell(u)\leq k-2$,
then $\rho_{G,\ell}(\{u,v\})\leq \lambda-k+3$. 
Then, by Lemma \ref{lem:j(k-2)},
$V(G) = \{u,v\}$.
So, by Observation \ref{obs:list-degree},
$\ell(v) \leq k-2$, so that $\rho_{G,\ell}(\{u,v\}) \leq 2((k-2)\lambda-1)-2(k-2)\lambda - (k-3)\leq-4$, a contradiction.
}

Thus $\ell(u)=k-1$, and by symmetry, $\ell(v)=k-1$.
By Observation \ref{obs:list-degree}, $d(u) \geq k-1$ and $d(v) \geq k-1$.
Also, as $\rho_{G,\ell}(\{u,v\})=2\lambda+5-k < 3(\lambda-2)+1$, it holds that
$d(u)=d(v)=k-1$.
In particular, writing $S = \{u,v\}$, we have $\|S, \overline S\| = 2$.

We write $G' = G-u-v$.
By the minimality of $G$,  $G'$ has an $(H,\LL)$-coloring $(\phi,\sigma)$.
By Lemma \ref{lem:partial}, $G[\{u,v\}]$ has no $(H,\LL_{\phi})$-coloring.
 Therefore, by part (iii) of Theorem \ref{KSSt},
one of the following holds:
\begin{equation}\label{418}
\parbox{15cm}{\begin{enumerate}[(a)]
    \item The nodes of $\LL_{\phi}(u) \cup \LL_{\phi}(v)$ 
all have capacity at most $1$, and the nodes of capacity $1$ induce a $K_{k-2,k-2}$ in $H$, or
    \item  $k=6$,
    all nodes of $\LL_{\phi}(u) \cup \LL_{\phi}(v)$ have capacity $0$ or $2$, and the nodes of capacity $2$ induce a $2K_{2,2}$ in $H$.
\end{enumerate}}
\end{equation}

In both cases, the nodes in $\LL_{\phi}(u) \cup \LL_{\phi}(v)$ are not adjacent to the other nodes in $\LL(u) \cup \LL(v)$.
Thus, 
\begin{equation}\label{Lphi-2}
   \parbox{14.5cm}{there is a unique node $(u,c_u)\in H(u)$ such that
$\ell_{\psi}(u,c_u) < \ell(u,c_u)$  for every $(H,\LL)$-coloring $(\psi,\tau)$ of $G'$.} 
\end{equation}

Write $u'$ for the unique neighbor of $u $ in $G-u-v$, and write $(u',c_{u'})$ for the unique neighbor of $(u,c_u)$ in $H(u')$. We define $(v,c_v)$, $v'$, and $(v',c_{v'})$ similarly. Note that we may have $u'=v'$. 
By~\eqref{Lphi-2},
\begin{equation}\label{ucu-2}
  \mbox{$\psi(u')=c_{u'}$ and $\psi(v')=c_{v'}$ for every $(H,\LL)$-coloring $(\psi,\sigma)$ of $G'$.}  
\end{equation}

\begin{claim}
\label{claim:u'v''} { Both $uu'$ and $vv'$ are cut edges.}
\end{claim}
\begin{proof}
    Suppose that $uu'$ is not a cut edge.
    Let $\LL'$ be obtained from $\LL$ by reducing $\ell(u',c_{u'})$ by $1$.
    We claim that $G'$ has no $(H,\LL')$-coloring.
    Indeed, if $G'$ has an $(H,\LL')$-coloring $(\psi,\tau)$,
    then by \eqref{ucu-2},
    $\psi(u') = c_{u'}$ and $\psi(v') = c_{v'}$.
    Then, we can extend $(\psi,\tau)$ to an $(H,\LL)$-coloring of $G$ as follows.
    First,
    extend $(\psi,\tau)$ by choosing a node $(v,c'') \in H(v)  \setminus N_H(v',c_{v'})$ for which $\ell_{\psi}(v,c'') > 0$, and then placing
$(v,c'')$ and $(u,c_u)$ (in this order) before $(\psi,\tau)$.
We show below that this gives us an $(H,\LL)$-coloring of $G$.

Write $H''$ for the subgraph of $H[\{u,v\}]$ induced by those nodes $\alpha$ satisfying $\ell_{\psi}(\alpha) > 0$.
By \eqref{418}, $H''$ is $(k-2)$-regular.
Note also that $\ell(u,c_u) = \ell_{\psi}(u,c_u)+1$.
Therefore,
if $\ell(u,c_u) = 1$, 
then $(u,c_u)$ is not adjacent to $(v,c'')$,
and hence $d^-_{\psi,\tau}(\alpha) < \ell(\alpha)$
for each $\alpha \in \psi$.
If $\ell(u,c_u) = 2$, then
as $\ell_{\psi}(u,c_u) = 1$,
 \eqref{418} implies that no two nodes of $H''$ are joined in $H$ by parallel edges.
 As $(u,c_u) \in H''$,
 this implies that $d^-_{\psi,\tau} (u,c_u) \leq 1 < \ell(u,c_u)$,
 which implies again that 
 $d^-_{\psi,\tau}(\alpha) < \ell(\alpha)$
for each $\alpha \in \psi$.
Therefore, $(\psi,\tau)$ is an $(H,\LL)$-coloring of $G$, a contradiction. Therefore, $G'$ has no $(H,\LL')$-coloring.

    Hence, there exists a subset $U \subseteq V(G')$ for  which 
    $G'[U]$ has a spanning $(H,\LL' )$-minimal subgraph. 
    As $G$ is $(H,\LL)$-minimal and is a minimum counterexample, $U$ contains $u'$, and $\rho_{G',\ell'}(U) \leq -2$.
    Therefore, $\rho_{G',\ell}(U) = \rho_{G,\ell}(U) \leq \lambda < 2(\lambda-2)+1$.
    Thus, by Lemma \ref{lem:j(k-2)},
    $\|U,\overline U\| = 1$.
    Since $uu' \in E_G(U,\overline U)$, it is a cut edge, contradicting our assumption. Similarly, $vv'$ is  a cut edge.
\end{proof}

By Claim~\ref{claim:u'v''}, $G\setminus\{uu',vv'\}$ has three components with vertex sets  $U_1, U_2$ and $\{u,v\}$, where $u'\in U_1$ and $v'\in U_2$. 
Let $G'' = G' + u'v'$, and let $H' = H + (u',c_{u'}) (v',c_{v'})$.
{ Note that $u'v'$ is a cut-edge in $G''$.}
By Lemma \ref{lem:cut-edge} and~\eqref{ucu-2},
$(u',c_{u'})$ is tight in 
every $(H,\LL)$-coloring of $G[U_1]$ and 
$(v',c_{v'})$ is tight in 
every $(H,\LL)$-coloring of $G[U_2]$.
Thus,
$G''$ has no $(H',L)$-coloring.
Therefore, there is a subset $S \subseteq V(G'')$ containing $u'$ and $v'$ such that  $\rho_{G'',\ell}(S) \leq -2$
or $(G''[S],\ell)$ is exceptional.
Since no exceptional graph has a cut-edge,
$\rho_{G'',\ell}(S) \leq -2$. Therefore,
$\rho_{G,\ell}(S) \leq 2\lambda-2$.
Then, $\rho_{G,\ell}(S \cup \{u,v\}) \leq -2$, a contradiction.
\end{proof}

\begin{lemma}
\label{lem:U-2Ct-subgraph}
Suppose $k=5$.
    If $G[U]$ has a spanning $2C_t^-$ subgraph, then (i) $\ell(u) = 4$ for each $u \in U$,  (ii) $G[U]$ is isomorphic to $2C_t^-$,
(iii) $\|U, \overline U\| = 3$, and (iv) $G-U$ is connected.
\end{lemma}

\begin{proof}
    If $G[U]$ has at least $2t+1$ edges, then $\rho(U) \leq 1-2\lambda < -2$, 
    a contradiction. If $G[U]$ has exactly $2t$ edges, then $\rho(U) \leq 1$, 
    so that $V(G) = U$. If $\ell(u)\leq 3$ for some $u \in U$, then $\rho(U) \leq -1-\lambda < -2$, a contradiction; therefore, 
     (i) holds. Since $G$ contains $2C_t^-$, has $2t$ edges and is not $2C_t$, $G$ contains a degree $3$ vertex, contradicting Observation \ref{obs:list-degree}.

    Therefore, $G[U]$ has exactly $2t-1$ edges, so (ii) holds.
    We write $x$ and $y$ for the two vertices of degree $3$ in $G[U]$.
    If  exactly $j$ vertices $u\in U$ satisfy
    $\ell(u) 
    \leq 3$, then 
    \begin{equation}\label{xyz}
   \mbox{     $\rho(U) \leq 2\lambda + 1 - j(\lambda+2)$, so $j\leq 1$. }
    \end{equation}
  
  Also, since $2C_t^-$ is a degree-choosable reducible subgraph,
  there is $z\in U$ with $d_G(z)>\ell(z)$ (possibly $z\in \{x,y\}$).
Then by Observation~\ref{obs:list-degree}, 
 \begin{equation}\label{xyz'}
  |E(\{x,y,z\},\overline{U})|\geq 3-j.   
    \end{equation}
 If $j=1$, then~\eqref{xyz'} together with~\eqref{xyz} contradicts 
 Lemma~\ref{lem:j(k-2)}. Hence $j=0$ and (i) holds.
 
Now by~\eqref{xyz}, $\rho_{G,\ell}(U) = 2\lambda+1 < 4(\lambda-2)+1$.
So,  by Lemma~\ref{lem:j(k-2)},
    $\|U, \overline U\| \leq 3$ and  by~\eqref{xyz'},  $\|U, \overline U\| \geq 3$. Thus (iii) holds.

Finally, suppose that $G-U$ is disconnected.
Let $S$ be the vertex set of a component of $G-U$ that is joined to $G[U]$ by $\lfloor \frac 32 \rfloor = 1$ edge.
By Lemma \ref{lem:k-1k}, $\rho_{G,\ell}(S) \leq \lambda$. Therefore, $\rho_{G,\ell}(S \cup U) \leq 2\lambda + 1 + \lambda - 2\lambda = \lambda + 1 < 2(\lambda-2)+1$.
Therefore,
by Lemma \ref{lem:j(k-2)},
$\|S\cup U, \overline{S \cup U}\| \leq 1$,
which contradicts the fact that $G[S]$ is joined to $U$ by a single edge and $\|U, \overline U\| = 3$. Thus, (iv) holds.
\end{proof}

\subsection{Boundaries of size $2$}

\begin{lemma}
\label{lem:reduce-one}
    Let $S \subseteq V(G)$ be a connected set satisfying $\|S, \overline S\| = 2$.
    For each $v \in S$ with a neighbor in $\overline S$,
    and for each node $(v,c) \in H(v)$,
   for some $(H,\LL)$-coloring $(\phi,\sigma)$ of $G[S]$ either $\phi(v) \neq c$ or
  $(v,c)$ is slack in $(\phi,\sigma)$.
\end{lemma}
\begin{proof}
Suppose for the sake of contradiction that
$(v,c)$ is tight
for every $(H,\LL)$-coloring
$(\phi,\sigma)$ of $G[S]$.
In this case, if $\LL'$ is obtained from $\LL$ by reducing $\ell(v,c)$ by $1$, then 
$G[S]$ has no $(H,\LL')$-coloring.
So, since $G$ is $(H,\LL)$-minimal, by~\eqref{minG}
there exists a subset $U \subseteq S$ satisfying $\rho_{G,\ell'}(U) \leq -2$ and hence $\rho_{G,\ell}(U) \leq \lambda < 2(\lambda-2)+1$.
Thus, by Lemma \ref{lem:j(k-2)},
$\|U,\overline U\| = 1$. 
Since $U$ has a neighbor in $\overline S$,
$\|U, S\setminus U\| = 0$.
As $G[S]$ is connected,  
$U=S$.
Since $\|S, \overline S\| = 2$, we have a contradiction.
\end{proof}

\begin{lemma}
\label{lem:2-bd-exception}
    Suppose that $S \subseteq V(G)$ satisfies $\|S, \overline S\| = 2$ and that there are two vertices $u,v \in S$ with neighbors in $\overline S$. Then, 
    \begin{enumerate}[(i)]
        \item $G[S]$ has no $2C_t^-$-subgraph containing $u$ and $v$;
        \item $G[S]$ has no $K_k^-$-subgraph containing $u$ and $v$; 
        \item If $(u,a)$ is slack in some $(H,\LL)$-coloring $(\psi,\tau)$ of $G[S]$, then $u$ precedes $v$ in $\tau$.
        \item For every node $\alpha \in H(u)$, there is at most one node $\beta \in H(v)$ for which an $(H,\LL)$-coloring of $G[S]$ contains both $\alpha$ and $\beta$.
        \item $\rho_{G,\ell}(S) \leq 2 \lambda-1$, and $\rho_{G,\ell}(S) \leq \lambda-2$ if $uv \not \in E(G)$.
    \end{enumerate}
\end{lemma}
\begin{proof}
    Proof of (i):     Suppose $U \subseteq \overline S$ contains $w$ and $x$ and induces a $2C_t^-$ subgraph of $G[\overline S]$.
     Then, by Lemma \ref{lem:U-2Ct-subgraph} (iii), $\|U, \overline U\| = 3$, so that $U \subsetneq \overline S$.
        Then, 
        no path in $G-U$ connects a vertex of 
        $G[\overline S \setminus U]$
        with a vertex of $G[S]$, 
        contradicting  Lemma \ref{lem:U-2Ct-subgraph} (iv).

     Proof of (ii): if $U \subseteq \overline S$
        contains $w$ and $x$ and induces a $K_k^-$,
        then by Lemma \ref{lem:KK-4}(i,ii),
        $\|U, \overline U\| = 4$,
        and $5\leq k\leq 6$. In particular, $U \subsetneq \overline S$.
        Then, 
        no path in $G-U$ connects a vertex of 
        $G[\overline S \setminus U]$
        with a vertex of $G[S]$, 
        contradicting  Lemma \ref{lem:KK-4}(iii).

        Proof of (iii): Suppose  that $(\phi',\sigma')$ is an $(H,\LL)$-coloring of $G[S]$ in which $(u,a)$ is slack and $u$ appears after $v$.
        Write $x$ for the neighbor of $v$ in $\overline S$ and $(x,c) \in H(x)$ for the neighbor of $(v,\psi(v))$.
    By Lemma \ref{lem:reduce-one}, there is an $(H,\LL)$-coloring
    $(\psi',\tau')$
    of $G[\overline S]$ in which $(x,c)$ is not tight.
    Then, if we take the union of $(\phi',\sigma')$ and $(\psi',\tau')$ with 
     $(\psi',\tau')$ placed between $v$ and $u$, then we have an $(H,\LL)$-coloring of $G$, a contradiction. 

     Proof of (iv): Let $\alpha = (u,a)$, and let $(\phi,\sigma)$ be an $(H,\LL)$-coloring of $G[S]$. Write $\phi(v) = b$.
     Suppose that there is an $(H,\LL)$-coloring $(\phi',\sigma')$ of $G[S]$ that includes both $(u,a)$ and some node $\beta' \in H(v) \setminus \{(v,b)\}$.

     Write $w \in \overline S$ for the neighbor of $u$, and write $(w,c) \in H(w)$ for the neighbor of $(u,a)$.
     Write $x \in \overline S$ for the neighbor of $v$ and write $(x,d) \in H(x)$ for the neighbor of $(v,b)$.
     Using Lemma \ref{lem:reduce-one}, let $(\psi,\tau)$ be an $(H,\LL)$-coloring of $G[\overline S]$ in which $(w,c)$ is not tight.
     If $\psi(x) = d$, then the union of $(\phi',\sigma')$ and $(\psi,\tau)$ in this order is an $(H,\LL)$-coloring of $G$, a contradiction. Otherwise, the union of $(\phi,\sigma)$ and $(\psi,\tau)$ is an $(H,\LL)$-coloring of $G$, again giving a contradiction. Thus, (iv) holds. 

     Proof of (v):
     Let $G' = G+uv$. Let $M$ be a matching joining $H(u)$ and $H(v)$ such that each $\alpha \in H(u)$ is joined to the unique $\beta \in H(v)$ specified in (iv), if $\beta$ exists.
     Then, let $H' = H\cup M$.

     Now, suppose that $G'[S]$ has an $(H,\LL)$-coloring $(\phi,\sigma)$.
     Then, $(\phi,\sigma)$ is also an $(H,\LL)$-coloring of $G[S]$, so by (iv), $(\phi,\sigma)$ contains two nodes $\alpha \in H(u)$ and $\beta \in H(v)$ that are joined by $M$.
     Without loss of generality, $v$ appears before $u$ in $\sigma$, so that $\beta$ is a left-neighbor of $\alpha$ in $(\phi,\sigma)$ applied to $G'$. Thus,
     \[
        d^-_{G,\phi,\sigma}(\alpha) < d^-_{G',\phi,\sigma}(\alpha) \leq \ell(\alpha) - 1,
     \]
     and so $\alpha$ is slack with respect to $(\phi,\sigma)$ on $G$.
     However, since $v$ precedes $u$ in $\sigma$, (iii) is contradicted. 
     Thus, $G$ has no $(H',\LL)$-coloring.

     By Lemma \ref{l42-2}
     and (i) and (ii), it follows that for some $U \subseteq S$ containing $u$ and $v$,
     $\rho_{G',\ell}(U) \leq -2$.
     Thus, $\rho_{G,\ell}(U) \leq 2\lambda-1 < 3(\lambda-2)+1$.
     As $\|U,\overline S\| = 2$, Lemma \ref{lem:j(k-2)} and the connectivity of $G[S]$ imply that $U = S$.
     Thus, $\rho_{G,\ell}(S) 
\leq 2\lambda-1$.
Finally, we note that $\rho_{G,\ell}(U) \leq 2\lambda-2$ whenever $uv \not \in E(G)$,
and then the same argument shows that $\rho_{G,\ell}(S) \leq 2\lambda-2$. Thus,
(v) holds.
\end{proof}

Now, we obtain the following structural lemma.

\begin{lemma}
\label{lem:boundary-2}
    If $S \subseteq V(G)$ 
    satisfies
    $\|S, \overline S\| = 2$, then:
    \begin{enumerate}
        \item \label{item:C2C3} The two edges of 
    $E_G(S,\overline S)$ 
    belong to a $C_2$ or $C_3$, or they are both cut-edges. 
    \item \label{item:2l-12l}
    If $G[S]$ and $G[\overline S]$ are connected, then
    $\rho_{G,\ell}(S), \rho_{G,\ell}(\overline S) \in \{2\lambda-1,2\lambda\}$.
    \item \label{item:single-bd} If $S$ has a single vertex on its boundary, then $\rho_{G,\ell}(S) = 2\lambda$.
    \end{enumerate}
\end{lemma}

 \begin{proof}
     Suppose that $S \subseteq V(G)$ satisfies $\|S,\overline S\|=2$.
     If $E_G(S,\overline S)$ consists of two cut-edges, then (\ref{item:C2C3}) holds.
     In this case,
     without loss of generality $G[\overline S]$ is
     not connected, so (\ref{item:2l-12l}) holds vacuously.
     
     Furthermore, we claim that $S$ has at least two vertices on its boundary. Indeed, if $S$ has only one vertex $u$ on its boundary, then letting $\phi$ be an $(H,\LL)$-coloring of $G[\overline S]$, Lemma \ref{lem:partial} implies that $G[S]$ has no $(H,\LL_{\phi})$-coloring. 
     Therefore, for some $U \subseteq S$ containing $u$, $\rho_{G,\ell_{\phi}}(U) \leq -2$, so that $\rho_{G,\ell}(U) \leq -2 + \lambda + (\lambda+2) = 2\lambda$.

     Now, there are disjoint subsets $S_1,S_2 \subseteq \overline S$
     such that $G[\overline S]$ has components $G[S_1]$ and $G[S_2]$.
     For each $i \in \{1,2\}$,
     there is a vertex $v_i \in S_i$ incident to a cut-edge.
     Given an $(H,\LL)$-coloring $(\psi,\tau)$ of $G[S]$, Lemma \ref{lem:partial} implies that $G[\overline S]$ has no $(H,\LL_{\psi})$-coloring, which means that for each $i \in \{1,2\}$,
     there is a unique node $(v_i,c_i) \in H(v_i)$ that is tight for every $(H,\LL)$-coloring of $G[S_i]$.
     Hence, if we define $G' = G[\overline S] + v_1v_2$ and define $H' = H[\overline S] + (v_1,c_1)(v_2,c_2)$,
     $G'$ has no $(H',\LL)$-coloring.
     Therefore, there is a set $W \subseteq \overline S$ containing $v_1$ and $v_2$ for which $\rho_{G',\ell}(W) \leq -2$ (note that $v_1v_2$ cannot belong to an exceptional subgraph of $G$, as no exceptional graph contains a cut-edge).
     Therefore, $\rho_{G,\ell}(W) \leq 2\lambda-2$.
     As $G[W]$ has two components, it follows without loss of generality that $\rho_{G,\ell}(W \cap S_1) \leq \lambda-1 < 2(\lambda-2)+1$.
     Therefore, by Lemma \ref{lem:j(k-2)}, $W \cap S_1) = S_1$ and  $\rho_{G,\ell}(S_1) = \lambda-1$.
     Hence, $\rho(W \cap S_2) \leq \lambda-1 < 2(\lambda-2)$, and hence again by Lemma \ref{lem:j(k-2)}, $W \cap S_2 = S_2$ and $\rho_{G,\ell}(S_2) = \lambda-1$.
     Then, 
     \[\rho_{G,\ell}(V(G)) = \rho_{G,\ell}(S) + \rho_{G,\ell}(S_1) + \rho_{G,\ell}(S_2) - 4\lambda \leq -2,\] 
     a contradiction.
     Therefore, $S$ has two vertices on its boundary, so (\ref{item:single-bd}) holds vacuously.

    For the rest of the proof, we assume that $E_G(S,\overline S)$ does not consist of two cut-edges, so that $G[S]$ and $G[\overline S]$ are connected.

     We consider three cases.

         {\bf Case 1:} There is a single vertex $v \in S$, and $v$ has a unique neighbor $w$ in $\overline S$.

         As
         $E_G(S,\overline S) = E_G(v,w)$,
         we see that
            $|E_G(v,w)|=2$, so (\ref{item:C2C3}) holds.
         By the minimality of $G$, we may choose an $(H,\LL)$-coloring $(\phi,\sigma)$ of $G[S]$.
        By Lemma \ref{lem:partial}, $G[\overline S]$ has no $(H,\LL_{\phi})$-coloring.
         Note also that $\ell_{\phi}(w) \geq \ell(w)-2$.
         Therefore, by the minimality of $G$, there is a set $U \subseteq \overline S$ containing $w$ and satisfying $\rho_{G,\ell_{\phi}}(U) \leq -2$.
         Thus, $\rho_{G, \ell}(U) \leq -2 + \lambda + (\lambda+2) = 2\lambda < 3(\lambda-2)+1$.
         Thus,
         it follows from Lemma \ref{lem:j(k-2)} that $U = \overline S$, so that $\rho_{G, \ell}(\overline S) \leq 2\lambda$.
         Therefore,
         \[
            -1 \leq \rho_{G, \ell}(G) = \rho_{G, \ell}(S) + \rho_{G, \ell}(\overline S) - (4\lambda+1) \leq \rho_{G, \ell}(S) - (2\lambda+1),
         \]
         implying that $\rho_{G, \ell}(S) \geq 2\lambda$.
         By applying the argument with $v$ and $w$ reversed,
         $\rho_{G, \ell}(S) \leq 2\lambda$, and $\rho_{G, \ell}(\overline S) \geq 2\lambda$. Therefore, $\rho_{G, \ell}(S) = \rho_{G, \ell}(\overline S) = 2\lambda$, so (\ref{item:2l-12l}) and (\ref{item:single-bd}) hold.

       {\bf Case 2: }
        There are two distinct vertices $u,v \in S$, each of which has a neighbor in $\overline S$.
        Note that (\ref{item:single-bd}) holds vacuously.
        We write $w \in \overline S$ for the neighbor of $u$, and we write $x \in \overline S$ for the neighbor of $v$.
        If $w=x$, then 
        letting $(\phi,\sigma)$ be an $(H,\LL)$-coloring
        of $G[\overline S]$, $G[\overline S]$ has no $(H,\LL_{\phi})$-coloring by Lemma \ref{lem:partial}.
        Thus, $\rho_{G,\LL_{\phi}}(U) \leq -2$ for some $U \subseteq \overline S$ containing $w$, so that $\rho_{G,\ell}(U) \leq 2\lambda < 3(\lambda-2)+1$.
        As $\|U,S\| =2$,
        Lemma \ref{lem:j(k-2)} implies that $U = \overline S$, and hence $\rho_{G,\ell}(\overline S) \leq 2\lambda$.

        Now,
        \begin{equation}
        \label{eqn:cherry}
-1 \leq \rho_{G,\ell}(V(G)) = \rho_{G,\ell}(S) + \rho_{G,\ell}(\overline S) - 4\lambda \leq \rho_{G,\ell}(S) - 2\lambda.
\end{equation}
        Thus, $\rho_{G,\ell}(S) 
        \geq 2\lambda-1$.
        Furthermore, by Lemma \ref{lem:2-bd-exception},
        $\rho_{G,\ell}(S) \leq 2\lambda-1$, with equality holding only if $uv \in E(G)$. Thus, $\rho_{G,\ell}(S)= 2\lambda-1$, so (\ref{item:2l-12l}) holds, and $uv \in E(G)$, so that $uw$ and $vw$ belong to a triangle, and thus (\ref{item:C2C3}) holds.
        Also, \eqref{eqn:cherry} must be tight, so $\rho_{G,\ell}(\overline S) = 2\lambda$.

        On the other hand, suppose that $w \neq x$.
        Then, applying Lemma \ref{lem:2-bd-exception} to both $S$ and $\overline S$, we see that $\rho_{G,\ell}(S) \leq 2\lambda-1$ and $\rho_{G,\ell}(\overline S) \leq 2\lambda-1$.
        Then,
        \[
        \rho_{G,\ell}(V(G)) = \rho_{G,\ell}(S) + \rho_{G,\ell}(\overline S) - 4\lambda \leq -2,
        \]
        a contradiction. 
        This completes the proof.

        \textbf{Case 3:} There is a single vertex $w \in S$ with two distinct neighbors $u,v \in \overline S$.
        Then, we apply the previous case with $S$ and $\overline S$ reversed, and we conclude that $u,v,w$ form a triangle, and that $\rho_{G,\ell}(S) = 2\lambda$. Therefore, (\ref{item:C2C3}), (\ref{item:2l-12l}), and (\ref{item:single-bd}) hold.
  \end{proof}

\begin{lemma}
\label{lem:no2}
    $G$ has no vertex $v$ satisfying $\ell (v) = 2$.
\end{lemma}

\begin{proof}
    Suppose that $v \in V(G)$ satisfies $\ell(v) = 2$.
    By Observation \ref{obs:list-degree},
    $d(v) \geq 2$.
    Since $\rho_{\ell}(v) = 2\lambda-1 < 3(\lambda-2) + 1$, Lemma \ref{lem:j(k-2)} implies that $d(v) = 2$.
    Hence, the set $\{v\}$ has exactly one vertex and two edges on its boundary, so Lemma \ref{lem:boundary-2} (iii) implies that $\rho_{\ell}(v) = 2\lambda$, a contradiction.
\end{proof}

\begin{cor}\label{cor:h-geq3}
    Every vertex $v$ in $G$ satisfies $\ell (v)\geq 3$.
\end{cor}
\begin{proof}
    This follows from Observation~\ref{obs:zero} and Lemmas~\ref{lem:no1} and~\ref{lem:no2}.
\end{proof}

\begin{lemma}
\label{lem:k-1-reg}
     Let $k\geq 5$.
    If $G$ is a block and $\ell (v) = k-1$ for each $v \in V(G)$, then $G$ is not $(k-1)$-regular.
\end{lemma}
\begin{proof}
 Suppose $G$ is $(k-1)$-regular. Then $|V(G)|\geq 2$ and by 
 Theorem~\ref{KSSt}, $G$ is $k$-constructible.

 If $|V(G)|=2$, then $G=(k-1)K_2$ and $\rho(V(G))=2-m(G)=4-k$. When $k\geq 6$, this is at most $-2$ contradicting~\eqref{minG}(c), and when $k=5$, this contradicts~\eqref{minG}(e).

  Suppose $G$ has at least $3$ vertices.
  Since $H$ is $2$-bounded, $(G,H)$ is not a monoblock.
 So, by~\eqref{minG}(d)
$G=q G'$ for some $q\geq 2$ and a block $G'$. 
Then $m(G)=(q-1)|E(G')|$, and so  
$$\rho(V(G)\leq |V(G)|-
 (q-1)|E(G')|=(1-q)(|E(G')|-|V(G)|)-
(q-2)|V(G)|.$$
By~\eqref{minG}(c), this at least $-1$, and so
 $q=2$  and $|E(G')|-|V(G)|\leq 1$.
Then $\delta(G')=2$. But then $k-1=\delta(G)=2\delta(G')=4$. Thus $k=5$
and $G$ is $4$-regular.
So, $G$ is a double cycle, contradicting~\eqref{minG}(e).
\end{proof}

\section{Eliminating  low ${(k-2)}$-regular blocks}
\label{sec:eliminating}

\subsection{On almost $(k-1)$-regular subgraphs of $G$}

By $K_t^-$ (respectively, $2C_t^-$) we denote the (multi)graph obtained from $K_t$ (respectively, from $2C_t$) by deleting one edge.

For $u,v\in V(G)$, let $K_k^-(u,v)$ denote a $K^-_k$-subgraph of $G$ in which $u$ and $v$ are the non-adjacent vertices and let $C_t^-(u,v)$ denote a $2C_t^-$-subgraph of $G$ in which $uv$ is the edge of multiplicity $1$.

\begin{lemma}\label{l32}
Suppose $k=5$.
If $U_1$ and $U_2$ are distinct vertex sets for which $G[U_1]$ has a spanning $2C_t^-$ subgraph and $G[U_2]$ has a spanning $2C_r^-$ subgraph,
then $U_1$ and $U_2$ are disjoint.
\end{lemma}
\begin{proof}
    Suppose that $U_1$ intersects $U_2$. Let $S = U_1 \setminus U_2$.
    As $2C_t^-$ is $3$-edge-connected,
    $\|S, U_2\| \geq 3$. 
    Furthermore, every edge cut of $2C_t^-$
    has a pair of parallel edges.
    Therefore, as $3 \leq \|S,U_2\| \leq \|U_2, \overline{U_2}\|$, $\rho(U_2) \geq 3(\lambda-2)+4 > 2\lambda+1 = \rho(U_2)$, a contradiction.
\end{proof}

\begin{lemma}\label{nl'}
Suppose $k=5$.
For any $v_1,v_2,w\in V(G)$, $G$ cannot contain a $2C_t^-(v_1,v_2)$
and a $K^-_5(v_1,w)$ at the same time. 
\end{lemma}
\begin{proof} 
Suppose $G$ has a copy $F$ of a $2C_t^-(v_1,v_2)$ is obtained from double cycle $v_1v_2\ldots v_tv_1$ by deleting one edge joining $v_1$ with $v_2$.
Then, suppose that $G$ also has a copy of $K^-_5(v_1,w)$ with a vertex set $U$, and let $U'=U\setminus \{v_1,w\}$.
Since $v_1$ is incident to a double
edge $v_1v_t$, by Observation~\ref{obs:6}, $d(v_1)\leq 5$.  
Since $v_1$ has $3$ neighbors in $F$ and three neighbors in $U'$,
it follows that $v_1$ has a neighbor $v \in V(F) \cap U'$, and hence $v_2 \in U'$ or $v_t \in U'$.
If $v_t\in U'$, then 
as $d(v_t) \leq 5$ by Observation \ref{obs:6},
$v_{t-1}\in U$. So $G[U]$ has at least two multiple edges, and $\rho_{G,\ell}(U)\leq 5(4\lambda+1)-11(2\lambda) = 5(25)-11(12)=-7$, a contradiction. 
Thus, $v_2\in U'$, and as $d(v_2)\leq 5$ by Observation \ref{obs:6}, $v_3\in U$. Again, this yields $v_4\in U$ and $G[U]$ has at least two multiple edges, a contradiction.
\end{proof}

\begin{lemma}\label{nl}
For any $u,v,w\in V(G)$, $G$ cannot at the same time contain a $K_k^-(u,v)$
and a $K_k^-(v,w)$ such that at least $3$ edges connect their union with the rest of $G$.
\end{lemma}
\begin{proof} Suppose $G$ contains a $K_k^-(u,v)$ with vertex set $U$ and
a $K_k^-(v,w)$ with vertex set $W$ such that $|E(U\cup W,V(G)\setminus (U\cup W)|\geq 3$. Let $X=U\cap W$ and $x=|X|$. Then $|U\cup W|=2k-x$,
and $|E(G[U\cup W])|\geq 2\left({k\choose 2}-1\right)-{x\choose 2}=
k(k-1)-2-\frac{x^2-x}{2}$. So
$$\rho_{G, \ell}(U\cup W)\leq (2k-x)(1+\lambda(k-1))-2\lambda\left(
k(k-1)-2-\frac{x^2-x}{2}\right)
 $$   
 $$=\lambda x^2-(1+k\lambda)x+2k+4\lambda=: g(x).
 $$
Since function $g(x)$ is quadratic in $x$ with a positive coefficient at $x^2$,
the maximum of $g(x)$ is achieved at $x=1$ or at $x=k-1$.
Calculating, we get $g(1)=(5-k)\lambda+2k-1$ and $g(k-1)=(5-k)\lambda+k+1<g(1)$.

On the other hand, since $|E(U\cup W,V(G)\setminus (U\cup W)|\geq 3$, by 
Lemma~\ref{lem:j(k-2)}, $\rho_{G, \ell}(U\cup W)\geq 3\lambda-5$. Let us show that $g(1)=(5-k)\lambda+2k-1<3\lambda-5$, which would imply the lemma.

Indeed, if $5\leq k\leq 6$, then $\lambda\geq k$, and $g(1)\leq 2\lambda-1<3\lambda-5$. If $k\geq 7$, then since $k<2\lambda$, we have
$g(1)\leq -2\lambda+(2k-1)<2\lambda-1<3\lambda-5
$, again.
\end{proof}

\subsection{Eliminating  low $(k-2)$-regular blocks}
We need a couple of definitions.

An \emph{edge-block} in a multigraph $G$ is an inclusion maximal connected subgraph $G'$ of $G$ such that
either $|V(G')|=2$ or $G'$ has no cut edges. In particular, every cut edge forms an edge-block, and each connected graph decomposes into edge-blocks.

Define a special subset $S^* \subseteq V(G)$ as follows. If $G$ has no cut edges, then $S^*=V(G)$. Otherwise, we fix a smallest pendent edge-block $B^*$ distinct from $K_2$
  and let $S^*=V(B^*)$. By Lemma~\ref{lem:no1} and Observation~\ref{obs:list-degree}, a cut edge cannot be a pendent edge-block, so $S^*$ is  well-defined.

    If $B^* \subsetneq G$, then
    since $B^*$ is pendent, there are $x^*\in S^*$ and $y^*\in V(G)-S^*$ such that $x^*y^*$ is the unique edge connecting $B^*$ with the rest of $G$. Fix these $B^*$, $S^*$, $x^*,y^*$.
  By definition, $B^*$ is $2$-edge-connected. 

Let $ \Lambda$ be the subgraph of $G$ induced by low vertices in $G$ and let ${ \Lambda}_0$ be the subgraph of $ \Lambda$ induced by the vertices of $ \Lambda$ in $S^*$.

Suppose $G$ has a $K_{k-1}$-block $K \subseteq  \Lambda$ with vertex set
$U=\{u_1,\ldots,u_{k-1}\}$ such that for some $0\leq j\leq k-1$,
$d(u_i)=k-2$ for  $1\leq i\leq j$ and
$d(u_i)=k-1$ for  $1+j\leq i\leq k-1$. For all $1+j\leq i\leq k-1$, let $x_i$ be the unique neighbor of $u_i$ outside of $U$, and let 
$X=\{x_{j+1},\ldots,x_{k-1}\}$.
Fix an $(H,\LL)$-coloring $(\phi,\sigma)$ of $G-U$.

\begin{obs}
\label{obs:extend22}
    We can rename the colors in $\LL_\phi[U]$ so that \\
    (i) for some $t\in [k-1]$ 
     and all $ i\in [k-1]$, ${\rm Supp}(u_i)=[t]$ and\\
    (ii) for all $c\in [t]$ and distinct $1\leq i,i'\leq k-1$, $(u_i,c)\sim (u_{i'},c)$ and $\ell_\phi(u_i,c)=\ell_\phi(u_{i'},c)$. 
\end{obs}
\begin{proof}
    Otherwise by  Theorem~\ref{KSSt}, $(\phi,\sigma)$ can be extended to $U$ by coloring $U$ last.
\end{proof}

 For $j+1\leq i\leq k-1$, let $(u_i,\gamma_i)$ be the node
  for which $\ell(u_i,\gamma_i)>\ell_\phi(u_i,\gamma_i)$,
 and let $(x_i,\mu_i)$ be the node such that $(u_i,\gamma_i)\sim (x_i,\mu_i)$. 

Given a $(k-2)$-regular block $B \subseteq  \Lambda$  and a vertex $u \in N(B) \setminus V(B)$ and $(u,\alpha)\in \LL(u)$, we write $F(B,u,\alpha)$ for the triple $(G_u,H_u,\LL_u)$, where $G_u=G\setminus B$, $H_u=H[G_u]$ and $\LL_u$ is obtained from $\LL$ by  decreasing
$\ell(u,\alpha)$
by $1$. 

Given a $(k-2)$-regular block $B \subseteq  \Lambda$  along with distinct vertices $x_i,x_{i'}\in N(B) \setminus V(B)$, we write $F(B,x_i,x_{i'})$ for the multigraph obtained from $G - V(B)$ by adding an edge $x_ix_{i'}$.

\begin{lemma}\label{nl3}  
{ Each vertex $u_i \in U$ satisfies $\ell(u_i) = k-1$.}
\end{lemma}

\begin{proof}
 Recall that $d(u_i)=k-2$ for  $1\leq i\leq j$ and
$d(u_i)=k-1$ for  $1+j\leq i\leq k-1$.
Suppose $1\leq j\leq k-1$. Since $G\neq K_{k-1}$ and is connected, $j\leq k-2$.
Let $G'=G - U$, and recall that $(\phi,\sigma)$ is an $(H,\LL)$-coloring of $G'$. 
Since at most one cut edge is incident to $U$, if $j\leq k-3$, we can rename the vertices in $U$ so that $u_{k-1}x_{k-1}$ is not a cut edge.
Let $\LL'$ be obtained from $\LL\vert_{G-U}$ by decreasing  $\ell(x_{k-1},\mu_{k-1})$ by $1$.

{\bf Case 1:}  $G'$ has an $(H',\LL')$-coloring, say
$(\phi',\sigma')$. 
 As $d(u_1) = \ell(u_1) = k-2$,  $\LL_{\phi'}(u_1)=\LL(u_1)$, so by Observation~\ref{obs:extend22}, $\phi'(x_{k-1})=\mu_{k-1}$.

If $(u_{k-1},\gamma_{k-1})$ has no neighbor
in ${\rm Supp}(u_1)$, then we place 
$u_{k-1}$ into $\sigma'$ before
$x_{k-1}$ and append $U-u_{k-1}$
at the end of $\sigma'$ in the order $u_{k-2},u_{k-3},\ldots,u_1$. Now, we let $\phi'(u_{k-1})=\gamma_{k-1}$ and color $U-u_{k-1}$ greedily from $\LL_{\phi'}$ following the order in $\sigma'$. Since 
$(u_{k-1},\gamma_{k-1})$ does not forbid any node in $\LL(u_1)$, we can finish the coloring, a contradiction.

{ On the other hand, suppose that
$(u_{k-1},\gamma_{k-1})$ has a neighbor in $\supp(u_1)$.
Then, by Observation \ref{obs:extend22},
$(u_{k-1},\gamma_{k-1})\sim (u_{1},\gamma_{k-1})$,
and $\ell(u_1,\gamma_{k-1})>0$.}
This means that $\ell(u_{k-1},\gamma_{k-1}) = 2$,
but 
$\ell_{\phi'}(u_i,\gamma_{k-1})=1$
for every  $1\leq i\leq k-1$.
 Then we place 
$u_1$ and $u_{k-1}$ (in this order)
into $\sigma'$ before $x_{k-1}$ and append $U\setminus\{u_1,u_{k-1}\}$
at the end of $\sigma'$ in any order. After that, we let $\phi'(u_{1})=\phi'(u_{k-1})=\gamma_{k-1}$ and color $U\setminus\{u_1,u_{k-1}\}$ greedily from $\LL_{\phi'}$. Since after coloring $G'$, $u_1$, and $u_{k-1}$,
the total remaining capacity of
all nodes in $H(u_i)$ is at least $k-3$
for each $2\leq i\leq k-2$, we can extend $\phi'$ to an $(H,\LL)$-coloring of $G$, a contradiction.

{\bf Case 2:}  $G'$ has no $(H',\LL')$-coloring. Let $G''$ be an $(H',\LL')$-minimal subgraph of $G'$, and let  $W=V(G')$. Since $G$ itself is $(H,\LL)$-minimal, $x_{k-1}\in W$. Since $G$ does not contain $K_k$ and   $\ell'(x_{k-1}) < k-1$, by the minimality of $G$, $\rho_{G',\ell'}(W)\leq -2$. Hence $\rho_{G,\ell}(W)\leq \lambda$.
Then  by Lemma~\ref{lem:j(k-2)}, $|E(W,V(G)\setminus W)|\leq 1$, and
$x_{k-1}u_{k-1}\in E(W,V(G)\setminus W)$. So, $x_{k-1}u_{k-1}$ is a cut edge; thus by the choice of $u_{k-1}$, $j=k-2$. In this case,
$$\rho_{G, \ell}(W\cup U)\leq \lambda+(k-2)((k-2)\lambda-1)+((k-1)\lambda+1)-\left ( \frac{k^2-3k+4}{2} \right ) 2 \lambda=3-k\leq -2,
$$
a contradiction.
\end{proof}

\begin{lemma}
     Suppose that $\ell_{\phi}(u_j,\gamma_j) = 0$ for some $j \in [k-1]$.
    Then, for all $i \in [k-1] \setminus \{j\}$, $(u_j,\gamma_j) \sim (u_i,\gamma_i)$, and in particular, $\ell_{\phi}(u_i,\gamma_i) = 0$.
\end{lemma}
 
\begin{proof}
Suppose
that $\ell_{\phi}(u_j,\gamma_j) = 0$.
As $\ell(u_j) = k-1$ for all $j$, we may assume that $j=1$.
 If $(u_1,\gamma_1)\not\sim (u_i,\gamma_i)$
 for some $i$, say 
 $(u_1,\gamma_1)\not\sim (u_{k-1},\gamma_{k-1})$, then we
 consider $(G',H_1,\LL_1)$, 
 where $G' = G-U$,
 $H_1=H[G']$ and $\LL_1$ is obtained from
 $\LL\vert_{G'}$ by decreasing   $\ell(x_1,\mu_1)$  by $1$.

 If $G'$ has  an 
$(H_1,\LL_1)$-coloring $(\phi_1,\sigma_1)$, then 
Claims (i) and (ii) of  Observation~\ref{obs:extend22} hold for  $(\phi_1,\sigma_1)$. 
In this case, we insert $u_1$ before
$x_1$ in $\sigma_1$,  vertices $u_2,\ldots,u_{k-1}$ at the end of 
$\sigma_1$ (in this order), and we 
let
 $\phi_1(u_1)=\gamma_1$. Then,
by (ii) and   our assumption, $\phi_1(u_1)$
 does not
 reduce the capacity of any node in ${\rm Supp}_{\LL_{\phi}}(u_{k-1})$.
 So, we can color
$u_2,\ldots,u_{k-1}$ (in this order) greedily from their lists. This produces an 
$(H,\LL)$-coloring of $G$, a contradiction.

 Therefore, $G'$ has no $(H_1,\LL_1)$-coloring. Thus,
$G'$ contains an $(H_1,\LL_1 )$-minimal subgraph, say with vertex set $W$. 
 Since $G$ does not contain $K_k$ and 
  $\ell_1(x_1) < k-1$, by the minimality of $G$,
$\rho_{G',\ell_1}(W)\leq -2$. Then,
by Observation~\ref{obs:geq-1}, $x_1\in W$ and  $\rho_{G,\ell}(W)\leq -2 + (\lambda+2) = \lambda < 2(\lambda-2)+1$.
Now,
by Lemma~\ref{lem:j(k-2)}, $\|W,\overline W \| \leq 1$, and $x_1u_1\in E_G(W,\overline W)$. Thus, $u_1=x^*$ and $x_1=y^*$.

Now, let $G_y$  denote the component of $G-U$ containing $y^*$,
and let $G_x=(G-U)\setminus G_y$. Let 
$(\phi_y,\sigma_y)$ denote the restriction of $(\phi,\sigma)$ to  $G_y$. Let $H_2=H[G_x]$ and $\LL_2$ be obtained from 
 $\LL[{G_x}] $ by decreasing the capacity of $(x_2,\mu_2)$ by $1$. Since $u_2x_2$ is not a cut edge, the argument of the previous paragraph implies that $G_x$ has an $(H_2,\LL_2)$-coloring $(\phi_2,\sigma_2)$.

If $\LL_{\phi_2}(u)\neq \LL_\phi(u)$ 
for at least one $u\in U-u_1$, then we can extend
$\phi\vert_{G_y}\cup \phi_2$ to $U$ by placing $U$ after $\sigma_y\sigma_2$ and using Observation~\ref{obs:extend22} together with the fact that the list of $u_1$ after coloring $x_1$ with $\mu_1$ is $\LL_\phi(u_1)$. So, suppose $\LL_{\phi_2}(u)= \LL_\phi(u)$ for all $u\in U-u_1$. Then, $\phi_2(x_2)=\mu_2$.
We now construct an $(H,\LL)$-coloring of $G$ as follows.
Consider the vertex ordering obtained from $\sigma_y\sigma_2$ by placing $(u_1,u_2)$ in front of $x_2$ and appending $U\setminus \{u_1,u_2\}$ at the end. Use $\phi_y$ on   $G_y$ and $\phi_2$ on  $G_x$  to produce a partial $(H,\LL)$-coloring $\phi'$.
If  $\ell(u_2,\gamma_2)=2$, then $\ell_{\phi'}(u_1,\gamma_2)\geq 1$,
and we let $\phi'(u_1)=\phi'(u_2)=\gamma_2$, and greedily color $U\setminus \{u_1,u_2\}$, since 
 $\ell_{\phi'}(u) \geq k-3$ for each $u \in U \setminus \{u_1,u_2\}$.
 On the other hand,
if  $\ell(u_2,\gamma_2)=1$, then  we let  $\phi'(u_2)=\gamma_2$, and greedily color $U\setminus \{u_2\}$, since
 $\ell_{\phi'}(u) \geq k-2$ for each $u \in U \setminus \{u_2\}$.
 In both cases, we find an $(H,\LL)$-coloring of $G$, giving a contradiction that proves the first part of the claim.

Finally, by Observation \ref{obs:extend22},
if $\ell_{\phi}(u_1,\gamma_1) = 0$ and $(u_i,\gamma_i) \sim (u_1,\gamma_1)$, 
then $\ell_{\phi}(u_i,\gamma_i) = 0$.
This completes the proof of the lemma.
 \end{proof}

Thus,
\begin{equation}\label{gamma5}
 \parbox{14.5cm}{if 
  $\ell_{\phi}(u_i,\gamma_i) = 0$
 for some $i\in [k-1]$, then 
 for all $j\in [k-1]\setminus \{i\}$,   $(u_j,\gamma_j)\sim (u_i,\gamma_i)$, and in particular, 
  $\ell_{\phi}(u_j,\gamma_j) = 0$}
\end{equation}

Together with Observation \ref{obs:extend22}, this yields that

 \begin{equation}
 \label{nnl}
 \parbox{14.5cm}{ $s(u_1)=s(u_2)=\ldots =s(u_{k-1})=:s$,
 and
we can rename the colors in the lists so that for all $i\in [k-1]$,
\begin{enumerate}[(a)]
\item 
${\rm Supp}(u_i)=[s]$;
\item 
$(u_i,\alpha)\sim (u_{i'},\alpha)$ for all $i'\in [k-1]\setminus \{i\}$ and $\alpha\in [s]$;
\item  if 
 $\ell_{\phi}(u_i,\gamma_i) = 0$
for some $i\in [k-1]$, then 
for each $\alpha\in [s]$,  
the values
$\ell(u_i,\alpha)$ are the same for all $i\in [k-1]$.
\end{enumerate}
}
\end{equation}

The following two lemmas  contradict to each other and hence contradict the existence of $U$.

\begin{lemma}\label{2a} 
 For each $1\leq i<i'\leq k-1$ with $x_{i'}\neq x_i$,
$F(B,x_i,x_{i'})$ is not  $(H,\LL) $-colorable.
\end{lemma}
\begin{proof} Suppose $F:=F(B,x_1,x_2)$ is   $(H,\LL)$-colorable. Rename the nodes in $\LL(U)$ so that 
the claims (a)-(c) in~\eqref{nnl} hold. 
For every $i\in [k-1]$ and $\alpha\in [s]$, denote by
$(x_i,\alpha')$ the node in $\LL(x_i)$ adjacent to 
$(u_i,\alpha)$ (if such a node exists).
Let the cover $H_F$ of $F$ be obtained from $H\vert_{G-U}$ by adding the matching $\{(x_1,\alpha')(x_2, \alpha'): \alpha\in [s]\}$, where 
the 
edge $(x_1,\alpha')(x_2,\alpha')$ exists if
and only if 
both $(x_1,\alpha') $ and $ (x_2,\alpha')$ exist.

By assumption, $F$ has an $H_F$-coloring $(\phi_F,\sigma_F)$.
  Then,
  Claims (i) and (ii) of Observation~\ref{obs:extend22} hold for $\LL_{\phi_F}[U]$. 
    For every $i\in [k-1]$, let $(u_i,\gamma_i)$ be the node in ${\rm Supp}(u_i)$ such that $\ell_{\phi_F}(u_i,\gamma_i)<\ell(u_i,\gamma_i)$.

By symmetry, we may assume that $x_2$ appears after $x_1$ in  $\sigma_F$. In this case, we put $u_1$ and $u_2$ (in this order) before
$x_2$ in $\sigma_F$
and 
 append
all vertices in $U\setminus\{u_1,u_2\}$ at the end of $\sigma_F$.
If $\ell(u_1,\gamma_1)=1$ or $\ell(u_2,\gamma_2)=1$, then  by~\eqref{nnl}(b) and (c),
 $\ell(u_1,\gamma_1) = \ell(u_2,\gamma_2)=1$  and 
$\gamma_1=\gamma_2$. 
 In this case, 
 { we assign $\phi_F(u_2) = \gamma_2$}
 and { then color} 
 vertices in $U \setminus \{u_2\}$ greedily from $\LL_{\phi_F}[U]$.
 On the other hand, if  $\ell(u_1,\gamma_1)=\ell(u_2,\gamma_2)=2$, then
  we color $u_1$ and
  $u_2$
 with $\gamma_2$,  and vertices in $U\setminus\{u_1,u_2\}$ greedily  from $[s]\setminus \{\gamma_2\}$.
 In both cases, we obtain an $(H,\LL)$-coloring of $G$, a contradiction.
\end{proof}

\begin{lemma}\label{nl2} 
There are distinct $x_i,x_{i'}\in X$ such that $F(B,x_i,x_{i'})$ is   $(H,\LL)$-colorable.
\end{lemma}

\begin{proof} Since $K_k\not\subset G$, $|X|\geq 2$.
Suppose $|X|=2$, say $X=\{x_1,x_2\}$. Consider $F:=F(B,x_1,x_{2})$.
If $F$ is not  $(H,\LL) $-colorable, then it contains an $(H,\LL)$-minimal subgraph $F'$. 
Since $G$ is $(H,\LL)$-minimal itself, $X\subset V(F')$.
By the minimality of $G$, $\rho_{F,\ell}(V(F'))\leq k$. 
So,
$\rho_{G, \ell}(V(F'))\leq k+2\lambda$. Then
$$\rho_{G,\ell}(U\cup V(F'))\leq k+2\lambda+ (k-1)(1+(k-1)\lambda)-\frac{k(k-1)}{2}2\lambda=2k-1-(k-3)\lambda.$$
Since $\lambda\geq 6$ and $k\geq 5$, this is at most
$2k-1-6(k-3)=17-4k\leq -3$, a contradiction.

Suppose now $|X|\geq 3$. By Lemma~\ref{nl}, we can choose distinct $x_i,x_{i'}\in X$ 
 such that $G\setminus U$ does not contain $K_k^-$ in which edge $x_ix_{i'}$ is missing. Moreover, if $k=5$ then by Lemmas~\ref{l42-2},~\ref{l32} and~\ref{nl'}, we can choose this pair so that in addition $G[\{x_i,x_{i'}\}]\neq 3K_2$, and for any $t$,
$G\setminus U$ does not contain a $2C_t^-$ in which the  edge $x_ix_{i'}$ has multiplicity $1$.

Again, consider $F:=F(B,x_i,x_{i'})$.
If $F$ is not  $(H,\LL) $-colorable, then it contains an $(H,\LL)$-minimal subgraph $F'$.
Since $G$ is $(H,\LL)$-minimal itself, $\{x_i,x_{i'}\}\subset V(F')$. Now, since $F'\neq K_k$ and for $k=5$, $F'$ is neither $4K_2$ nor a double cycle,
 $\rho_{F',\ell'}(V(F'))\leq -2$.
 So,
$\rho_{G,\ell}(V(F'))\leq 2\lambda-1$. Then
by Lemma~\ref{lem:j(k-2)}, $|E_G(V(F'),V(G)\setminus V(F'))|\leq 2$,
and $x_iu_i,x_{i'}u_{i'}\in E_G(V(F'),\overline{V(F')})$. Since the edges
$x_iu_i$ and $x_{i'}u_{i'}$ form a matching, this contradicts Lemma \ref{lem:boundary-2}.
\end{proof}

Lemmas~\ref{nl3},~\ref{2a} and~\ref{nl2} together prove that 
\begin{equation}\label{nok-1}
 \mbox{\it $G$ does not contain
a vertex subset $U\subseteq \Lambda$
with $G[U]=K_{k-1}$.}   
\end{equation}

The proof of the next lemma uses the ideas of the proof of Lemma~\ref{l42-2}.

\begin{lemma}\label{lem:cycle+clique}
    $G$ has  
    \begin{enumerate}[(a)]
        \item  no low $(k-2)$-regular blocks isomorphic to $qC_t$ for $q \in \{2,3\}$ and $t\geq 3$, and
        \item  no low $(k-2)$-regular blocks isomorphic to $2K_t$ for $t \in \{4,5,6\}$.
    \end{enumerate}
\end{lemma}

\begin{proof} We prove (a) and (b) simultaneously and call them \emph{the cycle version} and \emph{the clique version}, respectively. 

  Suppose that some $B\subseteq V(G)$ induces a $qC_t$ block for some $t\ge 3$ and $q \in \{2,3\}$ or a $2K_t$ block for some $t \in \{4,5,6\}$.
In the cycle version, $G[B]$ is $2q$-regular, and $k=2q+2$; in the clique version,
  $G[B]$ is $2(t-1)$-regular, and $k=2t$. 
    Let $b=\|B,\ol{B}\|$.
In the cycle version,   $\rho_{G,\ell}(B)=-qt+(\lambda+2)b$, and by Lemma~\ref{lem:j(k-2)},    
    $\rho_{G,\ell}(B)\geq b(\lambda-2)+1$, implying   $4b\ge qt+1\geq 7$, and hence
    $b\geq 2$.
In the clique version, $
        \rho_{G,\ell}(V(B)) = -t(1+\frac{t-1}{2}) + (\lambda+2) b$,
        and  by Lemma \ref{lem:j(k-2)}, 
        $\rho_{G,\ell}(V(B))\geq b(\lambda-2)+1,$ yielding
    $4b \geq \frac{t(t+1)}{2} + 1\geq \frac{5t}{2}+1 \geq 11$,
    and hence $b\geq 3$.

    We write $G' = G-B$.
    By the minimality of $G$,  $G'$ has an $(H,\LL)$-coloring $(\phi,\sigma)$.
    By Lemma \ref{lem:partial}, $G[B]$ has no $(H,\LL_{\phi})$-coloring.
Let $H'$ be the subgraph of $H[B]$ 
induced by the nodes $\alpha \in V(H[B])$ for which $\ell_{\phi}(\alpha) > 0$.

In the cycle version, 
  by Parts (v) and (vi) of Theorem \ref{KSSt},
every component $J$ of $H'$ is one of the following:
\begin{itemize}
    \item A $q$-blowup of $C_{2t}$ in which all nodes $\alpha$ satisfy $\ell_{\phi}(\alpha)=1$,
    \item A $q$-blowup of $C_t$ in which all nodes  $\alpha$ satisfy $\ell_{\phi}(\alpha)=1$,
    \item A $q$-blowup of $C_t$ in which all nodes $\alpha$ satisfy $\ell_{\phi}(\alpha)=2$,
    \item A $qC_t$ in which all nodes $\alpha$ satisfy $\ell_{\phi}(\alpha)=q$,
    \item A $qC_{2t}$ in which all nodes  $\alpha$ satisfy $\ell_{\phi}(\alpha)=q$.
\end{itemize}
In all of these cases, $J$ is $2q$-regular,
so $J$ has no neighbors in $H[B] \setminus J$.
Also, $\sum_{\alpha \in J}\ell_{\phi}(\alpha) \geq qt  \geq 2t$.

In the clique version, by Part (iv)  of Theorem \ref{KSSt},
every component $J$ of $H'$ is one of the following:
\begin{itemize}
    \item A $2$-blowup of $K_t$ in which all nodes $\alpha$ satisfy $\ell_{\phi}(\alpha)=1$,
    \item A $2$-blowup of $K_t$ in which all nodes $\alpha$ satisfy $\ell_{\phi}(\alpha)=2$,
    \item A $2K_t$ in which all nodes $\alpha$ satisfy $\ell_{\phi}(\alpha)=2$.
\end{itemize}
In all of these cases, $J$ is $2(t-1)$-regular,
so $J$ has no neighbors in $H[B] \setminus J$.
Also, $\sum_{\alpha \in J}\ell_{\phi}(\alpha) \geq 2t$.

Furthermore, for each $u \in B$ with $\ell(u)=k-1$, $\LL(u)$ is obtained from $\LL_{\phi}(u)$ by increasing the capacity of some node $(u,c_u)$ by $1$.
We consider cases based on $\ell(u,c_u)$:
\begin{itemize}
    \item If $\ell(u,c_u) = 1$, then as $\ell_{\phi}(u,c_u)=0$, $(u,c_u)$ does not belong to a component of $H'$.
    Therefore, $(u,c_u)$ is the unique node of $H(u)$ satisfying $\ell(u,c_u)>0$ and belonging to a component $J'$ of $H[B]$ for which $\sum_{\beta \in J'}\ell(\beta) \leq |B|=t <2 t$.
    \item If $\ell(u,c_u)=2$, then $(u,c_u)$ belongs to a component in $H[B]$ containing a second node $(u,c') \in H(u)$ satisfying $\ell(u,c') = 1$.
\end{itemize}
In both cases, Parts (iv)--(vi) of Theorem \ref{KSSt} imply that $\ell_{\psi}(u,c_u) < \ell(u,c_u)$ for every $(H,\LL)$-coloring $(\psi,\tau)$ of $G'$.
Thus, for every $u \in B$ with a neighbor $u' \in \overline B$,
 \begin{equation}\label{Lphi-5}
       \parbox{14.5cm}{there is a unique node $(u,c_u)\in H(u)$ such that
    $\ell_{\psi}(u,c_u) < \ell(u,c_u)$  for every $(H,\LL)$-coloring $(\psi,\tau)$ of $G'$.} 
    \end{equation}

We fix $u \in B$ with a neighbor $u' \in \overline B$, and we write $(u',c_{u'})$ for the unique neighbor of $(u,c_u)$ in $H(u')$.  By~\eqref{Lphi-5},
\begin{equation}\label{eqn:unique-color-5}
  \mbox{$\phi(u')=c_{u'}$  for every $(H,\LL)$-coloring $(\phi,\sigma)$ of $G'$.}  
\end{equation}

\begin{claim}
\label{claim:cut} {  $uu'$ is a cut edge.}
\end{claim}
\begin{proof}
    Let $\LL'$ be obtained from $\LL$ by reducing $\ell(u',c_{u'})$ by $1$.
    We claim that $G'$ has no $(H,\LL')$-coloring.
    Indeed, if $G'$ has an $(H,\LL')$-coloring $(\psi,\tau)$,
    then by \eqref{eqn:unique-color-5},
    $\psi(u') = c_{u'}$.
    We can extend $(\psi,\tau)$ to an $(H,\LL)$-coloring of $G$ as follows:

For each $v \in V(B) \setminus \{u\}$ with $\ell(v)=k-1$,
obtain $\LL'(v)$
from $\LL(v)$ 
by updating $\ell'(v,c_v) = 0$,
where $(v,c_v)$ is adjacent to the node $(v',c_{v'})$  defined in \eqref{eqn:unique-color-5}.
In the cycle version, $\ell'(v)\geq  2q-1>q$, and in the clique version, $\ell'(v)\geq k-3 > 2(t-2)$.
Now, we order $V(B)$ with $u$ on the right and in the cycle version let every other vertex have at most one left-neighbor. Observe that in the cycle version,
$u$ has left-degree $2q = k-2 < \ell(u) = \ell'(u)$ in $B$,
and each $v \in V(B) \setminus \{u\}$ has left-degree at most $q < 2q-1 \leq \ell'(v)$ in $B$.
Similarly, 
in the clique version, $u$ has left-degree $2(t-1) < k-1 = \ell'(u)$ in $B$,
and  each $v \in V(B) \setminus \{u\}$ has left-degree at most $2(t-2) < k-3 \leq \ell'(v)$.
Therefore $B$ has an $(H,\LL')$-coloring $(\psi',\tau')$.
We place $(\psi',\tau')$ before $(\psi,\tau)$ and claim that this new coloring $(\phi,\sigma)$ 
is an $(H,\LL)$-coloring of $G$.

Indeed, as $(u',c_{u'})$ was slack in $(\psi,\tau)$ and gained at most one left-neighbor, $d^-_{\phi,\sigma}(u',c_{u'}) < \ell(u',c_{u'})$.
For each $w \in V(G') \setminus \{u'\}$,
$d^-_{\phi,\sigma}(w,\psi(w)) =  d^-_{(\psi,\tau) }(w,\psi(w)) < \ell(w,\psi(w))$.
 In particular, for each $v' \in \overline B$ with a neighbor $v \in B$,
\eqref{eqn:unique-color-5}
and the fact that $\psi'(v) \neq c_v$ together imply that $(v,\psi'(v))$ and $(v',\psi(v'))$ are not adjacent.
Finally, as $(\psi',\tau')$ is an $(H,\LL)$-coloring of $B$, for each $v \in V(B)$, $d^-_{\phi,\sigma}(v,\phi(v)) < \ell(v,c)$.
Therefore, $(\phi,\sigma)$ is an $(H,\LL)$-coloring of $G$, a contradiction.
Therefore, $G'$ has no $(H,\LL')$-coloring.

    Hence, there exists a subset $U \subseteq V(G')$ for  which 
    $G'[U]$ has a spanning $(H,\LL' )$-minimal subgraph. 
    As $G$ is $(H,\LL)$-minimal and is a minimum counterexample, $U$ contains $u'$, and $\rho_{G',\ell'}(U) \leq -2$.
    Therefore, $\rho_{G',\ell}(U) = \rho_{G,\ell}(U) \leq -2+(\lambda+2) < 2(\lambda-2)+1$.
    Thus, by Lemma \ref{lem:j(k-2)},
    $\|U,\overline U\| = 1$.
    Since $uu' \in E_G(U,\overline U)$, it is a cut edge, as claimed.
\end{proof}
 By Claim~\ref{claim:cut}, every edge from $B$ to $\ol{B}$ 
     is a cut edge.
    As $b \geq 2$, we
    apply Lemma \ref{lem:lambda-1} to 
    every component attached to $B$ by a cut-edge, so that each such component has potential at most $\lambda-1$.
    Then in the  cycle version we have
    \[
        \rho_{G,\ell}(V(G)) \leq \rho_{G,\ell}(B) - b(\lambda+1) =-qt+(\lambda+2)b-(\lambda+1)=-qt+ b  \leq -t \leq -3,
    \]
    and in the cliques case
 \[
        \rho_{G,\ell}(V(G)) \leq \rho_{G,\ell}(B) - b(\lambda+1) = -t\left (1+\frac{t-1}{2}\right ) +  b      
        \leq - \frac 52 t + b \leq -\frac 32 t \leq -6.
    \]
In both cases we come to    
    a contradiction.
\end{proof}

\begin{lemma}\label{all}
    $G$ has no low $(k-2)$-regular blocks.
\end{lemma}
\begin{proof} Suppose $G$ has a low $(k-2)$-regular block $B$. By Lemma~\ref{KSSt},
   $B$ is a multiple of a complete graph or a cycle. By~\eqref{nok-1}, $B\neq K_{k-1}$.
Since $k\geq 5$, $B$ is not a cycle. Thus, $B=t B'$, where $t\geq 2$ and $B'$ is a complete graph or a cycle. In particular, $k\geq 6$ or $B=3K_2$ and $k=5$. The latter is not true by Lemma~\ref{l42-2}.

 Suppose that $B = qC_t$.
By Lemma \ref{lem:cycle+clique}, $q \geq 4$.
In particular, $k =2q+2$.
For each vertex $v \in V(B)$, write $\Phi(v) = \rho_{\ell}(v) - 2q\lambda - (q-1) - (\lambda-2)(d(v) - k+2)$.
Note that $\Phi(v) = 3 - (q-1) \leq 0$ when $\ell(v) = k-1$ and $\Phi(v) < -1$ when $\ell(v)=k-2$.
However, by Lemma \ref{lem:j(k-2)},
we have 
\[
1 \leq \rho_{G,\ell}(B) - (\lambda-2) \|V(B),\overline{V(B)}\| = \sum_{v 
\in B} \Phi(v) \leq 0,
\]
which is a contradiction.

Therefore, $B = tK_r$, where $t(r-1) = k-2$, $r \geq 4$, and $t \geq 2$.
In particular, $k\geq 8$.
For each vertex $v \in V(B)$, write 
\[\Phi(v) = \rho_{\ell}(v) - t(r-1)\lambda - \frac 12 (r-1)(t-1) - (\lambda-2)(d(v) - k+2).\]
Note that $\Phi(v) = 3 - \frac 12 (r-1)(t-1)$ when $\ell(v) = k-1$ and $\Phi(v) < -1$ when $\ell(v)=k-2$.

We claim that $\frac 12 (r-1)(t-1) = \frac 12 (k-2)(1- \frac 1t) \geq 3$. 
Indeed, 
if $k=8$, then as $r\geq 4$ and $t \geq 2$ and $t(r-1)=k-2$, we have $r=4$ and $t=2$.
Then $G$ has a $(k-2)$-regular $2K_4$ in its low set, contradicting Lemma \ref{lem:cycle+clique}.
If $k=9$, then $t(r-1)=7$, 
contradicting the assumption that $t\geq 2$ and $r \geq 4$.
If $k=10$, then as $t(r-1)=8$, we have $t=2$ and $r=4$.
Then, $B$ is a $(k-2)$-regular $2K_4$, contradicting Lemma \ref{lem:cycle+clique}.
If $k=11$, then as $t(r-1)=9$, we have $t=3$ and $r=4$.
Then, $\frac 12 (k-2)(1 - \frac 13) = \frac 13 \cdot 9 = 3$.
If $k=12$, then as $t(r-1)=10$, we have $t=2$ and $r = 5$.
Then, $B$ is a $(k-2)$-regular $2K_5$, contradicting Lemma \ref{lem:cycle+clique}.
If $k=13$, then as $t(r-1)=11$, the lower bounds on $t$ and $r$ give a contradiction.
If $k \geq 14$, then as $t \geq 2$, we have $\frac 12 (k-2)(1-\frac 1t) \geq \frac 14 (k-2) \geq 3$.
Therefore, for every vertex $v \in V(B)$, we have $\Phi(v) \leq 0$.

Now, by Lemma \ref{lem:j(k-2)},
we have 
\[
1 \leq \rho_{G,\ell}(B) - (\lambda-2) \|V(B),\overline{V(B)}\| = \sum_{v 
\in B} \Phi(v) \leq 0.
\]
This final contradiction completes the proof.
\end{proof}

\section{Discharging}\label{sec:discharging}

For a multigraph $F$, let $\wt{F}$ denote the \emph{underlying graph of $F$}, i.e. the graph from which $F$ is obtained by multiplying some edges.
For a vertex $v\in V(F)$, let $\wt{d}(v)$ denote $d_{\wt F}(v)$, the degree of $v$ in $\wt{F}$, which is equal to $|N_F(v)|$.
We will use the symbols $\Phi_k$, $F$, $\mu_\ell(T)$, and $m(T)$ from Lemma~\ref{GDP'}. 
By Corollary~\ref{cor:h-geq3}, each vertex $v \in V(G)$ satisfies $\ell(v) \geq 3$.

We show that $\rho_{G,\ell}(G) \leq -2$, proving that $G$ in fact is not a counterexample to Theorem~\ref{thm:v3}.
We use the following discharging procedure.
Recall that $\nu = \frac{k-2}{2k-7}$. 

\begin{enumerate}
    \item  For each $v \in V(G)$, the initial charge of $v$ is $\rho_{G,\ell}(v)$. 
    For each pair $uv$, where $u,v\in V(G)$ are joined by $t \geq 1$ edges, the initial charge of the pair $uv$ is
    $-t(2\lambda +1) + 1$. For each 
    pair $uv$ of non-adjacent  $u,v \in V(G)$, the initial charge of $uv$ is $0$.
    \item For each pair $uv$ of adjacent vertices in $G$, if $ t\geq 1$ edges connect $u$ with $v$, 
    the pair $uv$ receives charge
    $(t(2\lambda +1)-1)/2$ from each of $u$ and $v$. 
    \item Each non-low vertex $u \in S^*$ takes charge $\nu$ along each edge $e$ that joins $u$ to a low vertex $v\in S^*$.
\end{enumerate}

For each $v \in  V(G)$, we write $ch^*(v)$ for the final charge of $v$.
Observe that the total charge in $G$ is $\rho_{G,\ell}(G)$. Additionally, the final charge of each vertex pair is $0$, so $\rho_{G,\ell}(G) =\sum_{v\in V(G)}ch^*(v)$.
Finally,
if $S^* \neq V(G)$, then the total charge in $G[\overline {S^*}]$ is at most 
$\rho_{G,\ell}(\overline {S^*}) - \lambda  \leq 0$ by Lemma~\ref{lem:k-1k}.
Therefore,
\begin{equation}
\label{eqn:ch-sum}
\rho_{G,\ell}(G) = \sum_{v \in V(G)} ch^*(v) \leq \sum_{v \in S^*} ch^*(v).
\end{equation}

For each $v \in S^*$, if $ d(v) \geq  1+\ell(v)$, then
we consider several cases.

\begin{enumerate}
    \item[(N1)] \label{item:small-nonlow} 
    If $d(v) \leq k-1$, then $\ell(v) \leq k-2$, so 
    \[ch^*(v) \leq (\ell(v)\lambda  - 1) - d(v) \lambda  + \nu d(v) = \lambda (\ell(v) - d(v)) - 1 + \nu d(v) \leq -\lambda -1 + \nu d(v) < -1.\]
    \item[(N2)] If $d(v) = k$, then
    \[ch^*(v) \leq ((k-1)\lambda +1) - k\lambda  + \nu k = -\lambda +1+\nu k \leq 0.\]

    \item[(N3)] If $d(v) \geq  k+1$, then 
    \begin{eqnarray*}
    ch^*(v) &\leq&  ((k-1)\lambda +1) - d(v) \lambda  + \nu d(v) \\ &\leq &
    ((k-1)\lambda +1) - (k+1) \lambda  + \nu (k+1)
    = 
    (-\lambda +1+\nu k) - \lambda  + \nu
    < -1. 
    \end{eqnarray*}
\end{enumerate}

For each $v \in S^*$, if $d(v)=\ell(v)=j$ for some fixed $j \in [3,\ldots, k-1]$, then we consider two cases:

\begin{enumerate}
\item[(L1)] If $j \leq k-2$, then
\begin{equation*}
\label{eqn:phi-neg}
ch^*(v)\leq (j\lambda -1)-j\lambda -
 m(v)- \nu ( d_{B^*}(v)-d_{{ \Lambda}_0}(v))
= -1-m(v)- \nu  d_{B^*}(v)+\nu d_{{ \Lambda}_0}(v).
\end{equation*}
\item[(L2)]  If $j=k-1$, then 
$$ch^*(v) \leq (k-1)\lambda +1-(k-1)\lambda -
 m(v)-\nu(d_{B^*}(v)-d_{{ \Lambda}_0}(v))=1-
 m(v)-\nu d_{B^*}(v)+ 
\nu d_{{ \Lambda}_0}(v).
$$
\end{enumerate}

We claim that ${ \Lambda}_0$ is nonempty. Indeed, if ${ \Lambda}_0$ is empty, then
as $d(v) \geq 1+\ell(v)$ for each $v \in S_0^*$, instead of  (N1)--(N3) we have the following inequality: 
\[ch^*(v) \leq ( \ell(v)\lambda  + 1) - d(v) \lambda  = (\ell(v) - d(v))\lambda  + 1 \leq -\lambda +1 < -2.\]
Then, by (\ref{eqn:ch-sum}), 
$\rho_{G,\ell}(G) < -2 |S^*| \leq   -2$, so $G$ is not a counterexample.
Thus, we assume that ${ \Lambda}_0 \neq \emptyset$, or equivalently, that $S_0^*$ contains at least one low vertex.

\begin{lemma}\label{charg2}
If $B$ is a component of ${ \Lambda}_0$,
then $\sum_{v\in V(B)}ch^*(v) < -1$.
\end{lemma}

\begin{proof} 
{ Let $V_{k-1}(B)$ be the set of vertices $v\in V(B)$ with $\ell(v)=k-1$ and $V_{k-1}^-(B)=V(B)\setminus V_{k-1}(B)$.
}Note that  $d_{B^*}(v)=\ell(v)$ for $v\in V(B) \setminus \{x^* \}$, and $d_{B^*}(x^*) = \ell(x^*) - 1$ whenever $x^* \in V({ \Lambda}_0)$.
Hence, by (L1)--(L2),
$$\sum_{v\in V(B)}ch^*(v)\leq |V_{k-1}(B)|-|V_{k-1}^-(B)|-\sum_{v\in V(B)}
m(v)- \nu \sum_{v \in V(B)} \left (\ell(v)-d_{{ \Lambda}_0}(v) \right)
 + \nu $$
$$
= |V_{k-1}(B)|-|V_{k-1}^-(B)|-m(B)- \nu \mu_\ell(B) + \nu  = -\Phi_k(B) + \nu.$$
The $\nu$ term accounts for the possibility that $x^* \in V(B)$, in which case $-d_{B^*}(x^*) = -\ell(x^*) + 1$.

By Corollary~\ref{cor:h-geq3}, Condition (i) of 
Lemma~\ref{GDP'} holds for $T=B$. Condition (ii) holds because the vertices in $B$ are low.
   By the minimality of $G$, Lemmas~\ref{lem:k-1-reg}
  and~\ref{l42-2} and by \eqref{nok-1}, Condition (iii) holds.
 Thus, by Lemma~\ref{GDP'},
$\sum_{v\in V(B)}ch^*(v)\leq -\Phi_k(B) + \nu < -1$.
\end{proof}

Now, by (N1)--(N3), the vertices of $S_0^* \setminus { \Lambda}_0$ have total nonpositive charge. 
Therefore,
\[\rho_{G,\ell}(G) =  \sum_{v \in V(G)} ch^*(v) \leq \sum_{v \in S^*} ch^*(v) \leq \sum_{v \in { \Lambda}_0} ch^*(v).\]
As ${\Lambda}_0 \neq \emptyset$, Lemma~\ref{charg2} implies that 
$\sum_{v \in { \Lambda}_0} ch^*(v) < -1$. 
Therefore, $\rho_{G,\ell}(G) < -1$, and as $\rho_{G,\ell}(G)$ is integral, $\rho_{G,\ell}(G) \leq -2$.
This completes the proof.

\section{List vertex arboricity and ordinary arboricity}
\label{sec:list-arboricity}
 For convenience, let us restate Theorem~\ref{t-lva}:

\begin{thm}\label{t-lva'}
    Let $k' \geq 3$, and let $G$ be an $\lva$-$k'$-critical graph. Then, either $G$ is a $K_{2k'-1}$, or 
   $\quad |E(G)| \geq (2k'-1+\frac{1}{\lambda'} ) \frac{|V(G)|}{2} + \frac 1{\lambda'},\quad$
    where $\lambda' = \left \lceil \frac{(2k'-1)^2 - 7}{4k'-9} \right \rceil $.
\end{thm}   
\begin{proof}  If $G$ has  a $K_{2k'-1}$ subgraph, then by criticality, $G= K_{2k'-1}$,
and the theorem holds.
Suppose $G$ has no $K_{2k'-1}$ subgraph. Since $G$ is simple, $G$ contains neither a $4K_2$ nor a double cycle.

    Let $L'$ be a $(k'-1)$-list assignment
    such that every assignment  of a color $\phi(v) \in L'(v)$ to each vertex of $G$ creates a monochromatic cycle.
    Write $L'(v) = (c_{v,1}, \dots, c_{v,k'-1})$ for each $v \in V(G)$.

 Construct a variable DP-cover $(H,\LL)$ for $G$ as follows.
    For each $v \in V(G)$, let  $\LL(v) = [2, \dots, 2]$ with $k'-1$ entries,  and 
    let $H$ be the cover for $G$ such that for each edge $uv \in E(G)$, the nodes $(u,i)$ and $(v,j)$ are adjacent in $H$ if and only if $c_{u,i} = c_{v,j}$.

    Let $k = 2k'-1$. Observe that $\lambda  = \left \lceil \frac{k^2 - 7}{2k-7}\right \rceil= \left \lceil \frac{(2k'-1)^2 - 7}{2(2k'-1)-7}\right \rceil=\lambda'$ and that
    $\ell(v) = k-1$ for each $v \in V(G)$. 
    Furthermore, $G$ has no $(H,\LL)$-coloring.  
    Choose $U$ to be a largest subset of $V(G)$ satisfying 
    $\rho_{G,\ell}(U) \leq -2$ if such a subset exists, and let $U = \emptyset$ otherwise.
    
    If $U = V(G)$, then the theorem holds. Therefore, suppose that $U \subsetneq V(G)$.
    Since $G$ is $\lva$-$k'$-critical, there exists an $L'$-coloring of $G[U]$, which yields an $(H,\LL)$-coloring $(\phi,\sigma)$ of $G[U]$. 
    By Lemma \ref{lem:partial},
    $G[\overline U]$ has no $(H,\LL_{\phi})$-coloring.
    So, by Theorem \ref{thm:v3},
    there exists a subset $U' \subseteq \overline U$ such that $\rho_{G,\ell_{\phi}}(U') \leq -2$.
    Letting $j = \|U',U\|$, it follows that $\rho_{G,\ell}(U') \leq -2 + j(\lambda+2)$.
    Then,
    \[
        \rho_{G,\ell}(U \cup U') \leq \rho_{G,\ell}(U) - 2j\lambda - 2 + j(\lambda+2) \leq -2 +j(2-\lambda) \leq -2,
    \]
    contradicting the maximality of $U$. 
    Therefore, $U = V(G)$, and the theorem holds.  \end{proof}

 The proof of  Theorem~\ref{t-va} repeats the proof of Theorem~\ref{t-lva'}
with $L'(v)=\{1,\ldots,k'-1\}$ for every $v\in V(G)$, so we omit it.

\paragraph{\bf Acknowledgment.} We thank Dan Cranston and Thomas Schweser for helpful comments.

\bibliographystyle{abbrv}
{\small
\bibliography{ref,ref2}}
  \end{document}